\documentclass[10pt]{amsart}
\usepackage{amsfonts}
\usepackage{amsmath, amscd, amssymb, graphicx}

\numberwithin{equation}{section}
\newtheorem{theorem}{Theorem}[section]

\theoremstyle{remark}
\newtheorem{remark}{Remark}[section]
\theoremstyle{definition}

\begin{document}
\title{Recursion Formulas for HOMFLY and Kauffman Invariants}
\author{Qingtao Chen}
\address{Q. Chen, Mathematics Section, International Center for Theoretical
Physics, Strada Costiera, 11, Trieste, I-34151, Italy \\
(qchen1@ictp.it)}
\author{Nicolai Reshetikhin}
\address{N. Reshetikhin, Department of mathematics, University of
California, Berkeley, CA 94720-3840, USA \\
(reshetik@math.berkeley.edu)}

\begin{abstract}
In this note we describe the recursion relations between two parameter
HOMLFY and Kauffman polynomials of framed links These relation correspond to
embeddings of quantized universal enveloping algebras. The relation
corresponding to embeddings $g_{n}\supset g_{k}\times sl_{n-k}$ where $g_{n}$
is either $so_{2n+1}$, $so_{2n}$ or $sp_{2n}$ is new.
\end{abstract}

\maketitle

\section*{Introduction}

\label{sec1}

The relation between knot polynomials and solutions to the Yang-Baxter
equation corresponding to quantized universal enveloping algebras was first
established in \cite{J} for HOMFLY and in \cite{T} for Kauffman polynomials.
In these references HOMFLY and Kauffman polynomials for special values of
the second parameter were written as a local state sum on a diagram of a
link. This result was extended in \cite{R} and \cite{RT} into the
construction of polynomial invariants of tangles graphs based on the
representation theory of quantized universal enveloping algebras.

Embeddings of simple Lie algebras which embed the Dynkin diagram of a Lie
algebra into the Dynkin diagram of the other algebra induce embeddings of
corresponding quantized universal enveloping algebras. For classical Lie
algebras we have:
\begin{equation}  \label{em-1}
U_q(sl_{n-k}) \otimes U_q(sl_{k})\subset U_q(sl_{n})
\end{equation}
\begin{equation}  \label{em-2}
U_q(sl_k)\otimes U_q(so_{2n-2k})\subset U_q(so_{2n}), \ \ U_q(sl_k)\otimes
U_q(sp_{2n-2k})\subset U_q(sp_{2n})
\end{equation}
\begin{equation}  \label{em-3}
U_q(sl_k)\otimes U_q(so_{2n+1-2k})\subset U_q(so_{2n+1})
\end{equation}
for $k=1,\dots, n-1$, and
\begin{equation}  \label{em-4}
U_q(sl_n)\subset U_q(so_{2n}), \ \ U_q(sl_n)\subset U_q(sp_{2n}), \ \
U_q(sl_n)\subset U_q(so_{2n+1})
\end{equation}

For a classical Lie algebra the restriction of the defining fundamental
representation (vector representations) of quantized universal enveloping
algebras to a diagrammatically embedded subalgebra gives certain identities
between HOMFLY and Kauffman polynomials when the second variable in these
polynomials is specialized to the appropriate power of $q$. These relations
were described in \cite{R} for embeddings (\ref{em-1}) and (\ref{em-4}) and
also for $sl_{2}\subset g_{2}$.

These relations were extended to non-specialized HOMFLY polynomial (when the
second variable is not specialized to a power of $q$) by F. Jaeger \cite{Ja}%
. They were generalized to invariants of graphs colored by special
representations by Kauffman and Vogel \cite{KV} and by Hao Wu \cite{Wu1} and
\cite{Wu2}.

The goal of this note is to give a complete list of such identities for two
variable HOMFLY and Kauffman polynomials. One should expect that each of the
functors constructed in this note can be categorified, see \cite{Wu1} and
\cite{Wu2} for some results in this direction.

In the first section we construct the functor $\phi _{tq;t,q}$ from the
category of HOMFLY skein modules $\mathcal{H}{}_{q;tq}$ to the category ${}%
\mathcal{H}_{q;t,q}$ which can be regarded as a shuffle tensor product of
additive categories ${}\mathcal{H}_{q;t}$ and ${}\mathcal{H}_{q;q}$. This
functor gives recursive relation for HOMFLY polynomials of framed links
expressing $H_{q,tq}(\mathcal{L})$ in terms of a linear combination of $%
H_{q,t}(\mathcal{L}^{\prime })$ for links $\mathcal{L}^{\prime }$ which are
obtained from $\mathcal{L}$ by a simple combinatorial procedure. This
recursion formula corresponds to (\ref{em-1}) with $k=1$. In the second
section we give the recursion formula corresponding to (\ref{em-1}) for all $%
k$ for HOMFLY polynomials as a functor from one skein category to another.
Similarly, the recursion relations from section 3 correspond to embeddings (%
\ref{em-4}). All this is an overview of \cite{Ja} and of \cite{R}. In
section 4 we give the recursion formula for Kauffman polynomials
corresponding to the embeddings (\ref{em-2}) and (\ref{em-3}) which is the
main result of this paper.

\section{Recursion corresponding to $sl_{n+1}\supset sl_{n}$}

\label{sec2}

\subsubsection{The functor $\protect\phi _{q;t,q}$}

The definition of HOMFLY polynomials via skein relations naturally extends
to invariants of tangles with values in skein modules \cite{T}. Let us
recall this construction.

Choose a line $L\subset \mathbf{R}^{2}$. A tangle $\mathcal{T}$ in $I\times
\mathbf{R}^{2}$ with end points in $L\subset (\{0\}\times \mathbf{R}%
^{2})\cup (\{1\}\times \mathbf{R}^{2})$ is an equivalence class of an
embedding $(S^{1})^{l}\times I^{k}\subset I\times {}{}{}{}\mathbf{R}^{2}$
such that end points belong to $(\{0\}\times {}{}{}{}\mathbf{R}^{2})\cup
(\{1\}\times {}{}{}{}\mathbf{R}^{2})$. Such embedding before taking the
equivalence class are called \textit{geometric tangles}. The equivalence is
taken with respect to homeomorphisms trivial at the boundary. A framing can
be thought as a homeomorphism class of a continuous section of the normal
bundle to $\mathcal{T}$. The framing is called blackboard if the
corresponding framed tangle has a representative where the framing at the
endpoints of $\mathcal{T}$ lies in the intersection of $L^{\perp }$ and $%
(0,0,1)^{\perp }$.

A \textit{diagram} of a framed tangle is the projection of a geometric
tangle to $L\times [0,1]$, assuming that the tangle has a blackboard framing.

\textit{Objects} of the category of tangles $\underline{Tan}$ are sequences $%
(\epsilon_1,\dots, \epsilon_n)$ with $\epsilon=\pm 1$.

\textit{Morphisms} between $\{\epsilon \}$ and $\{\sigma \}$ are oriented
framed tangles with the orientation which agrees with $\{\epsilon \}$ and $%
\{\sigma \}$ at the end points and such that the framing at the end points
is orthogonal to $L\subset {}{}{}\mathbf{R}^{2}$ and and point to positive
(with respect to the standard orientation of $L$ direction). Such framing is
called blackboard framing.

Note that morphisms in this category can be naturally identified with framed
Redemeister classes of diagrams of tangles.

The composition of morphisms is the gluing and then taking the homeomorphism
class of the result of the gluing. The detail can be found in \cite{T}.

For a ring $A$, define the additive category $\underline{Tan}_{A}$ as the
category with objects being direct sums of objects of $\underline{Tan}$ but
with morphisms being $A-$linear combination of tangles.

HOMLFY invariants are morphisms in quotient category of the $\underline{Tan}%
_{{}{}\mathbf{C}[t^{\pm 1},q^{\pm 1}]}$ subject to the following relations:

\begin{equation}
\raisebox{-0.2006in}
{\includegraphics[height=0.4687in,width=0.5197in]
{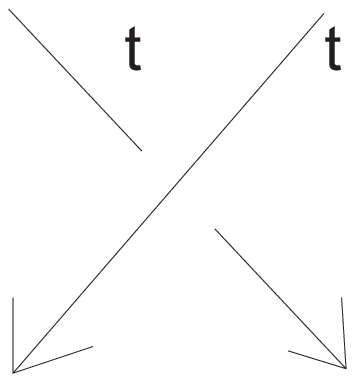}}\text{ }-\text{ }%
\raisebox{-0.2006in}
{\includegraphics[height=0.4687in,width=0.5197in]
{positivecross-tt.eps}}=(q-q^{-1})\text{ }%
\raisebox{-0.1609in}
{\includegraphics[height=0.5076in,width=0.4947in]
{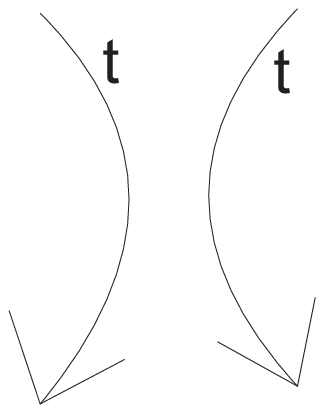}}  \label{HOMFLY}
\end{equation}%
and

\begin{equation}
\raisebox{-0.2006in}
{\includegraphics[height=0.6348in,width=0.4142in]
{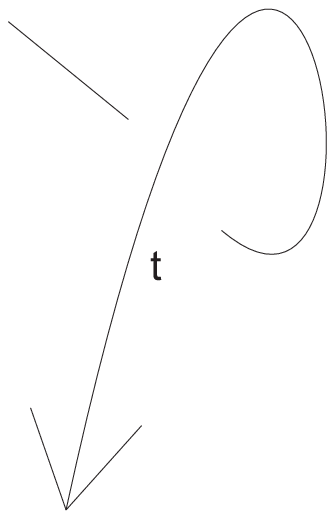}}\text{ }=t\text{ }%
\raisebox{-0.2214in}
{\includegraphics[height=0.6244in,width=0.2084in]
{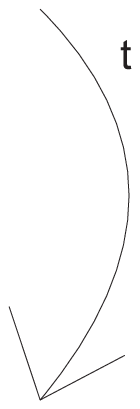}}\ ,\text{ }%
\raisebox{-0.1903in}
{\includegraphics[height=0.6547in,width=0.4047in]
{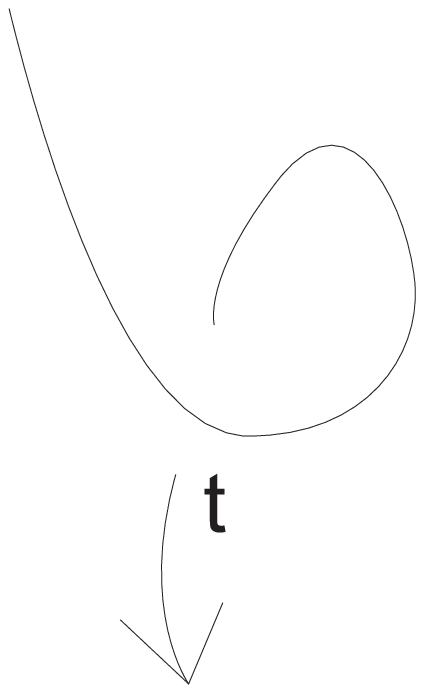}}\text{ }=t^{-1}\text{ }%
\raisebox{-0.2006in}
{\includegraphics[height=0.6244in,width=0.2084in]
{resolvekink-t.eps}}
\end{equation}

Objects of the resulting quotient category $\mathcal{H}{}_{q;t}$ are the
same as objects in the category of oriented framed tangles, i.e. sequences $%
(\epsilon _{1},\ \dots ,\epsilon _{n})$ where $\epsilon _{i}=\pm 1$.
Morphisms in this category are skein classes of linear combinations of
tangles. To indicate that the parameter involved in twist relations is $t$
we will write $t$ as a color of each connected component of of a tangle and
will write $((\epsilon _{1},t),\ \dots ,(\epsilon _{n},t))$ for objects.

\begin{remark}
The morphism between two identity objects (between two empty sequences)
given by an unknot with trivial framing is determined by the skein relations
and is equal to $(t-t^{-1})/(q-q^{-1})$.
\end{remark}

\begin{remark}
When $t=q^{n}$ the category $\mathcal{H}{}_{q;t}$ has the quotient category
which is naturally equivalent to the category of modules over $U_{q}(sl_{n})$%
.
\end{remark}

Define the category $\mathcal{H}{}_{q;s_{1},\dots ,s_{k}}$ as an additive
braided monoidal category over ${}{}{}\mathbf{C}[q^{\pm 1},s_{1}^{\pm
1},\dots ,s_{k}^{\pm 1}]$ with objects $((\epsilon _{1},s_{i_{1}}),\dots
,(\epsilon _{n},s_{i_{n}}))$. Morphisms between two such objects $A$ and $B$
are quotients of linear combination of tangles with blackboard framing whose
orientation agrees with the signs of $A$ and $B$, whose colors agree with
colors $\{s_{i_{a}}\}$ of $A$ and $B$. The quotient is taken with respect to
the HOMFLY relations which are (\ref{HOMFLY}) when both components are
colored by the same variable $s_{i}$; when $i\neq j$, we impose

\begin{equation}
\raisebox{-0.2006in}
{\includegraphics[height=0.4687in, width=0.5197in]
{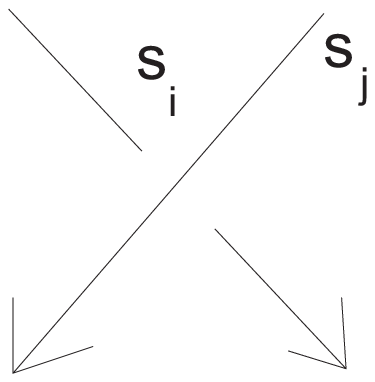}}\text{ }=\text{ }%
\raisebox{-0.2006in}
{\includegraphics[height=0.4687in, width=0.5197in]
{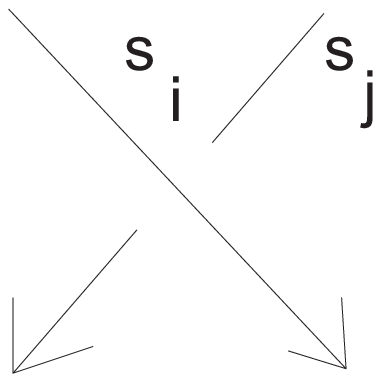}}\text{ }=\text{ }%
\raisebox{-0.2006in}
{\includegraphics[height=0.4687in, width=0.5197in]
{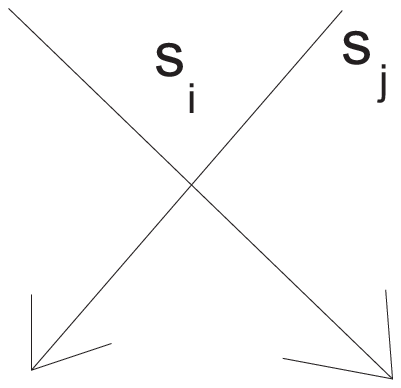}}
\end{equation}

Define the mapping $\phi _{q;t,q}:{}\mathcal{H}_{q;tq}\rightarrow \mathcal{H}%
{}_{q;t,q}$ which acts on objects as $\phi _{q;t,q}((\epsilon _{1},t),\dots
(\epsilon _{n},t))=\oplus ((\epsilon _{1},u_{1}),\dots ,(\epsilon
_{n},u_{n}))$. Here the sum is taken over all values of $u_{i}$ which are
either $q$ or $t$ an the summation is over all such possibilities. On
elementary morphisms we define $\phi _{q;t,q}$ as follows.

\begin{equation}
\phi _{q;t,q}\left(
\raisebox{-0.1609in}
{\includegraphics[height=0.4687in,width=0.5197in]
{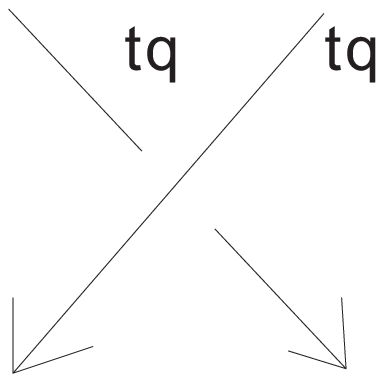}}\right) =\left(
\begin{array}{cccc}
\raisebox{-0.2006in}
{\includegraphics[height=0.4687in,width=0.5197in]
{positivecross-tt.eps}} & 0 & 0 & 0 \\
0 & (q-q^{-1})%
\raisebox{-0.2006in}
{\includegraphics[height=0.5076in, width=0.4947in]
{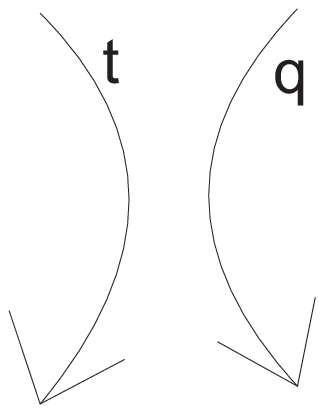}} &
\raisebox{-0.2006in}
{\includegraphics[height=0.4687in, width=0.5197in]
{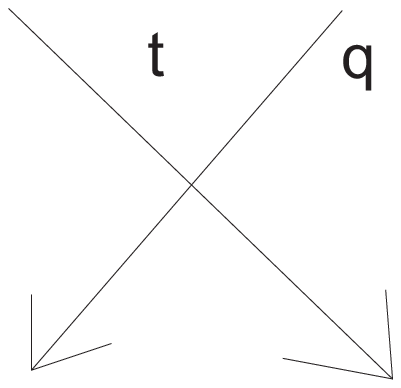}} & 0 \\
&  &  &  \\
0 &
\raisebox{-0.2006in}
{\includegraphics[height=0.4687in, width=0.5197in]
{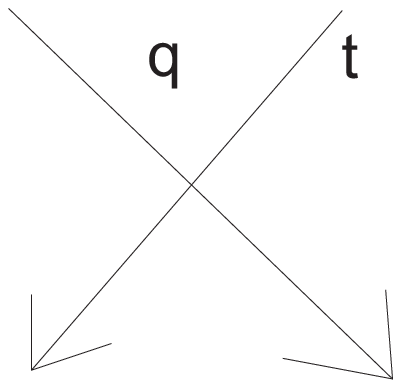}} & 0 & 0 \\
0 & 0 & 0 &
\raisebox{-0.2006in}
{\includegraphics[height=0.4687in,width=0.5197in]
{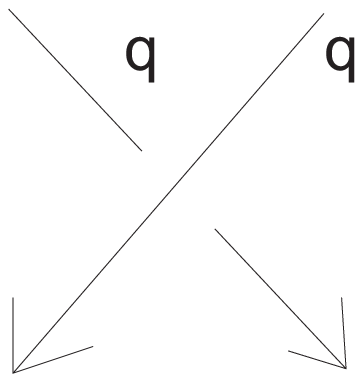}}%
\end{array}%
\right)
\end{equation}

\begin{equation}
\phi _{q;t,q}\left(
\raisebox{-0.1609in}
{\includegraphics[height=0.4679in,width=0.5197in]
{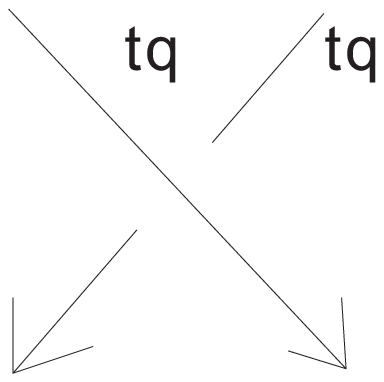}}\right) =\left(
\begin{array}{cccc}
\raisebox{-0.2006in}
{\includegraphics[height=0.4687in,width=0.5197in]
{positivecross-tt.eps}} & 0 & 0 & 0 \\
0 & 0 &
\raisebox{-0.2006in}
{\includegraphics[height=0.4687in, width=0.5197in]
{cross-tq.eps}} & 0 \\
&  &  &  \\
0 &
\raisebox{-0.2006in}
{\includegraphics[height=0.4687in, width=0.5197in]
{cross-qt.eps}} & -(q-q^{-1})%
\raisebox{-0.2006in}
{\includegraphics[height=0.5076in,width=0.4947in]
{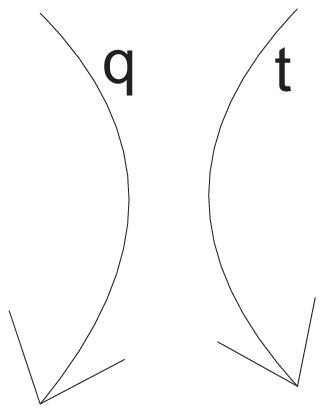}} & 0 \\
0 & 0 & 0 &
\raisebox{-0.2006in}
{\includegraphics[height=0.4679in,width=0.5197in]
{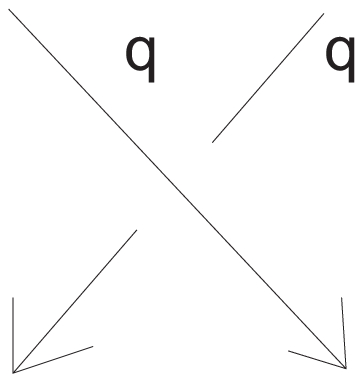}}%
\end{array}%
\right)
\end{equation}%
\begin{equation}
\phi _{q;t,q}\left(
\raisebox{-0.1609in}
{\includegraphics[height=0.4756in,width=0.2811in]
{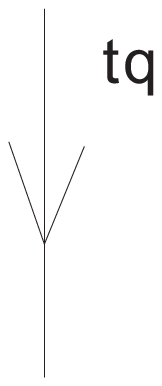}}\right) =\left(
\begin{array}{cc}
\raisebox{-0.2006in}
{\includegraphics[height=0.4765in,width=0.2776in]
{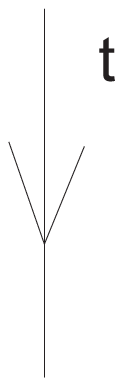}} & 0 \\
0 &
\raisebox{-0.2006in}
{\includegraphics[height=0.4765in,width=0.2776in]
{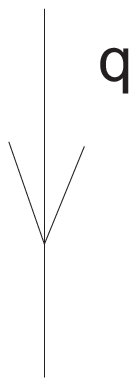}}%
\end{array}%
\right)
\end{equation}%
\begin{equation}
\phi _{q;t,q}\left(
\raisebox{-0.0502in}
{\includegraphics[height=0.1643in,width=0.3624in]
{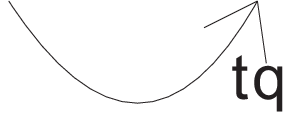}}\right) =\left(
\begin{array}{c}
{\includegraphics[height=0.1513in, width=0.3442in]
{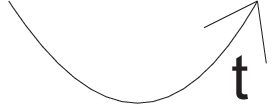}} \\
0 \\
0 \\
{\includegraphics[height=0.1643in, width=0.3442in]
{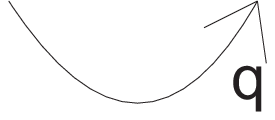}}%
\end{array}%
\right) \text{, }\phi _{q;t,q}\left( {%
\includegraphics[height=0.1643in,width=0.3624in ]
{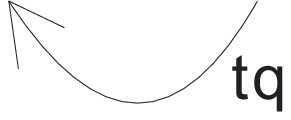}}\right) =\left(
\begin{array}{c}
q\text{ }%
\raisebox{-0.0502in}
{\includegraphics[height=0.1513in,width=0.3312in]
{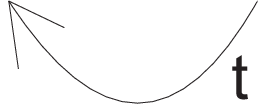}} \\
0 \\
0 \\
t^{-1}\text{ }%
\raisebox{-0.0502in}
{\includegraphics[height=0.1643in,width=0.3355in]
{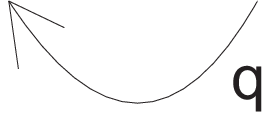}}%
\end{array}%
\right)
\end{equation}%
\begin{equation}
\phi _{q;t,q}\left(
\raisebox{-0.0502in}
{\includegraphics[height=0.1574in,width=0.2646in]
{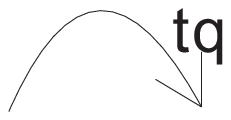}}\right) =\left(
\begin{array}{cccc}
\raisebox{-0.0502in}
{\includegraphics[height=0.1574in,width=0.339in]
{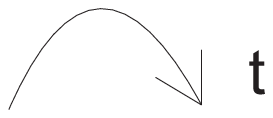}} & 0 & 0 &
\raisebox{-0.0502in}
{\includegraphics[height=0.1574in,width=0.2646in]
{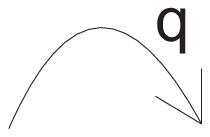}}%
\end{array}%
\right) \text{, }\phi _{q;t,q}\left(
\raisebox{-0.0502in}
{\includegraphics[height=0.1574in,width=0.3874in]
{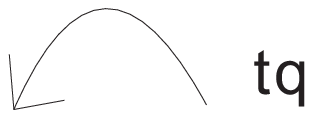}}\right) =\left(
\begin{array}{cccc}
q^{-1}\text{ }%
\raisebox{-0.0502in}
{\includegraphics[height=0.1565in,width=0.2984in]
{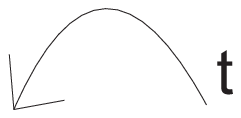}} & 0 & 0 & t\text{ }%
\raisebox{-0.0502in}
{\includegraphics[height=0.1574in,width=0.3113in]
{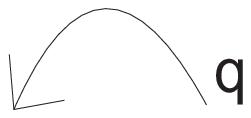}}%
\end{array}%
\right)
\end{equation}%
\begin{equation}
\phi _{q;t,q}\left(
\raisebox{-0.1609in}
{\includegraphics[height=0.5076in,width=0.4947in]
{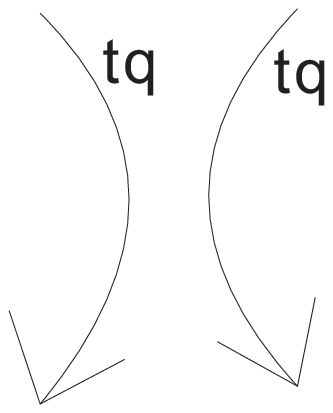}}\right) =\left(
\begin{array}{cccc}
\raisebox{-0.1609in}
{\includegraphics[height=0.5076in,width=0.4947in]
{resolvecorss-tt.eps}} & 0 & 0 & 0 \\
0 &
\raisebox{-0.1609in}
{\includegraphics[height=0.5076in,width=0.4947in]
{resolvecorss-tq.eps}} & 0 & 0 \\
0 & 0 &
\raisebox{-0.1609in}
{\includegraphics[height=0.5076in,width=0.4947in]
{resolvecorss-qt.eps}} & 0 \\
0 & 0 & 0 &
\raisebox{-0.1609in}
{\includegraphics[height=0.5076in, width=0.4947in]
{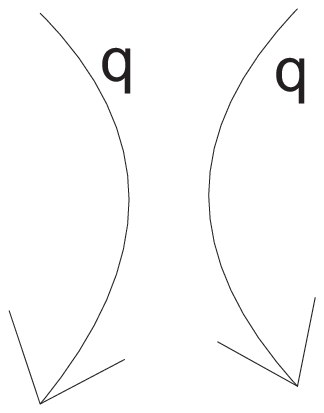}}%
\end{array}%
\right)
\end{equation}

Numerical factors are always to the left of the diagram and in a bigger
script.

The matrices on the right side indicate the morphisms between various
summands in the image of objects.

\begin{theorem}
The mapping $\phi _{q;t,q}$ extends uniquely to a covariant functor of
braided monoidal categories ${}\mathcal{H}_{q;tq}\rightarrow \mathcal{H}%
{}_{q;t,q}$.
\end{theorem}

\begin{proof}
Composing elementary diagrams we extend the mapping $\phi _{q;t,q}$ to all
objects. It is clear that if it is consistent with HOMFLY relations, then
such extension exists and unique. For the positive twist relation the
consistency is a simple computation:
\begin{align}
\phi _{q;t,q}\left(
\raisebox{-0.2006in}
{\includegraphics[height=0.6348in,width=0.4142in]
{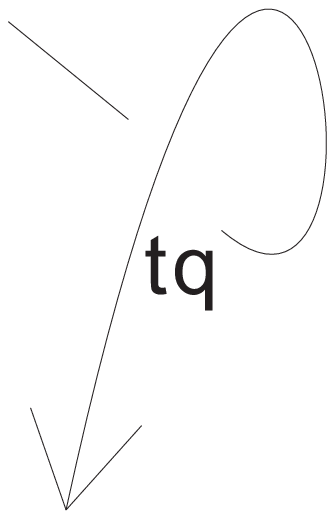}}\right) & =\phi _{q;t,q}\left(
\raisebox{-0.1609in}
{\includegraphics[height=0.4679in,width=0.5197in]
{positivecross-tqtq.eps}}_{%
\raisebox{-0.1003in}
{\includegraphics[height=0.1652in,width=0.3615in]
{rightlower-tq.eps}}}^{{%
\includegraphics[height=0.1565in,width=0.3874in]
{lefttop-tq.eps}}}\right) \\
& =\left(
\begin{array}{cc}
q^{-1}\text{ }%
\raisebox{-0.1609in}
{\includegraphics[height=0.4679in,width=0.5197in]
{positivecross-tt.eps}}_{%
\raisebox{-0.1003in}
{\includegraphics[height=0.1513in,width=0.3442in]
{rightlower-t.eps}}}^{%
\raisebox{0in}
{\includegraphics[height=0.1565in,width=0.2984in]
{lefttop-t.eps}}}+t\text{ }%
\raisebox{-0.1609in}
{\includegraphics[height=0.5076in,width=0.4947in]
{resolvecorss-tq.eps}}_{%
\raisebox{-0.1003in}
{\includegraphics[height=0.1643in,width=0.3442in]
{rightlower-q.eps}}}^{%
\raisebox{0in}
{\includegraphics[height=0.1574in,width=0.3113in]
{lefttop-q.eps}}}(q-q^{-1}) & 0 \\
0 & t\text{ }%
\raisebox{-0.1609in}
{\includegraphics[height=0.4679in,width=0.5197in]
{positivecross-qq.eps}}_{%
\raisebox{-0.1003in}
{\includegraphics[height=0.1643in,width=0.3442in]
{rightlower-q.eps}}}^{%
\raisebox{0in}
{\includegraphics[height=0.1574in,width=0.3113in]
{lefttop-q.eps}}}%
\end{array}%
\right) \\
& =\left(
\begin{array}{cc}
q^{-1}\text{ }%
\raisebox{-0.2006in}
{\includegraphics[height=0.6348in,width=0.4142in]
{positivekink-t.eps}}+t(q-q^{-1})%
\raisebox{-0.1609in}
{\includegraphics[height=0.5007in,width=0.6737in]
{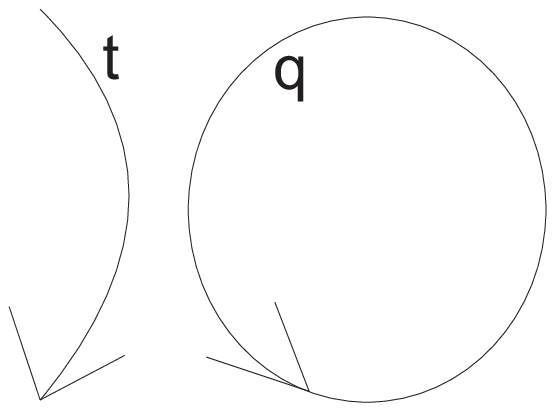}} & 0 \\
0 & t\text{ }%
\raisebox{-0.2006in}
{\includegraphics[height=0.6348in,width=0.4142in]
{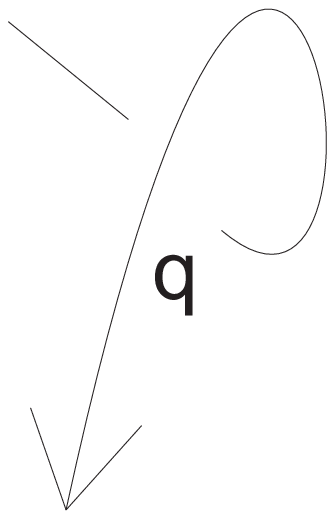}}%
\end{array}%
\right) \\
& =\left(
\begin{array}{cc}
q^{-1}t\text{ }%
\raisebox{-0.2214in}
{\includegraphics[height=0.6244in,width=0.2084in]
{resolvekink-t.eps}}+t(q-q^{-1})\text{ }%
\raisebox{-0.2214in}
{\includegraphics[height=0.6244in,width=0.2084in]
{resolvekink-t.eps}} & 0 \\
0 & tq\text{ }%
\raisebox{-0.2214in}
{\includegraphics[height=0.6244in,width=0.2084in]
{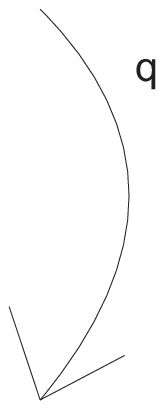}}%
\end{array}%
\right) \\
& =\left(
\begin{array}{cc}
tq\text{ }%
\raisebox{-0.2006in}
{\includegraphics[height=0.4765in,width=0.2776in]
{vertical-t.eps}} & 0 \\
0 & tq\text{ }%
\raisebox{-0.2006in}
{\includegraphics[height=0.4765in,width=0.2776in]
{vertical-q.eps}}%
\end{array}%
\right) =tq\cdot \phi _{q;t,q}\left(
\raisebox{-0.1609in}
{\includegraphics[height=0.4756in,width=0.2811in]
{vertical-tq.eps}}\right)
\end{align}%
Here we used the normalization $\phi _{q;t,q}\left(
\raisebox{-0.1609in}
{\includegraphics[height=0.4947in,width=0.4601in]
{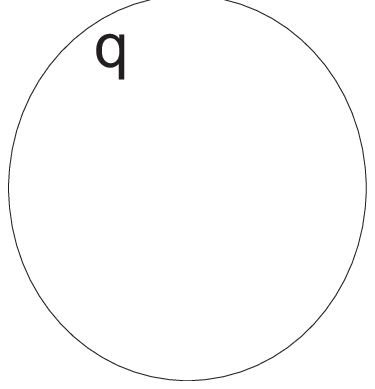}}\right) $ $=1$, which agrees with skein relations.

The consistency of the negative twist relations with the functor $\phi
_{q;t,q}$ is an almost identical calculation. The consistency of the skein
relation with the functor $\phi _{q;t,q}$ is given below:

\begin{equation*}
\phi _{q;t,q}\left(
\raisebox{-0.1003in}
{\includegraphics[height=0.3208in,width=0.3563in]
{positivecross-tqtq.eps}}\right) -\phi _{q;t,q}\left(
\raisebox{-0.1003in}
{\includegraphics[height=0.3217in,width=0.3563in]
{negativecross-tqtq.eps}}\right)
\end{equation*}%
\begin{equation*}
=\left(
\begin{array}{cccc}
\raisebox{-0.1003in}
{\includegraphics[height=0.3208in,width=0.3563in]
{positivecross-tt.eps}}-%
\raisebox{-0.1003in}
{\includegraphics[height=0.3217in,width=0.3563in]
{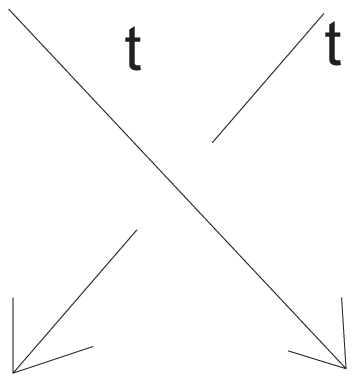}} & 0 & 0 & 0 \\
0 & (q-q^{-1})%
\raisebox{-0.1003in}
{\includegraphics[height=0.3477in,width=0.3381in]
{resolvecorss-tq.eps}} & 0 & 0 \\
0 & 0 & (q-q^{-1})%
\raisebox{-0.1003in}
{\includegraphics[height=0.3477in,width=0.3381in]
{resolvecorss-qt.eps}} & 0 \\
0 & 0 & 0 &
\raisebox{-0.1003in}
{\includegraphics[height=0.3208in,width=0.3563in]
{positivecross-qq.eps}}-%
\raisebox{-0.1003in}
{\includegraphics[height=0.3217in,width=0.3563in]
{negativecross-qq.eps}}%
\end{array}%
\right)
\end{equation*}%
\begin{equation*}
=\left(
\begin{array}{cccc}
(q-q^{-1})%
\raisebox{-0.2006in}
{\includegraphics[height=0.5076in,width=0.4947in]
{resolvecorss-tt.eps}} & 0 & 0 & 0 \\
0 & (q-q^{-1})%
\raisebox{-0.2006in}
{\includegraphics[height=0.5076in,width=0.4947in]
{resolvecorss-tq.eps}} & 0 & 0 \\
0 & 0 & (q-q^{-1})%
\raisebox{-0.2006in}
{\includegraphics[height=0.5076in,width=0.4947in]
{resolvecorss-qt.eps}} & 0 \\
0 & 0 & 0 & (q-q^{-1})%
\raisebox{-0.2006in}
{\includegraphics[height=0.5076in,width=0.4947in]
{resolvecorss-qq.eps}}%
\end{array}%
\right)
\end{equation*}

\begin{equation*}
=(q-q^{-1})\left(
\begin{array}{cccc}
\raisebox{-0.2006in}
{\includegraphics[height=0.5076in,width=0.4947in]
{resolvecorss-tt.eps}} & 0 & 0 & 0 \\
0 &
\raisebox{-0.2006in}
{\includegraphics[height=0.5076in,width=0.4947in]
{resolvecorss-tq.eps}} & 0 & 0 \\
0 & 0 &
\raisebox{-0.2006in}
{\includegraphics[height=0.5076in,width=0.4947in]
{resolvecorss-qt.eps}} & 0 \\
0 & 0 & 0 &
\raisebox{-0.2006in}
{\includegraphics[height=0.5076in,width=0.4947in]
{resolvecorss-qq.eps}}%
\end{array}%
\right) =\phi _{q;t,q}\left(
\raisebox{-0.1609in}
{\includegraphics[height=0.5076in,width=0.4947in]
{resolvecorss-tqtq.eps}}\right)
\end{equation*}
\end{proof}

\subsubsection{The recursion relation for invariants of links}

Let $\mathcal{L}$ be an oriented framed link and $D_{\mathcal{L}}$ be its
diagram with the blackboard framing. The image $[D_{\mathcal{L}}]_{q,t}$ of $%
D_{\mathcal{L}}$ in the skein module $Hom_{{}\mathcal{H}_{q;t}}(1,1)$, where
$1$ is the identity object (empty sequence), is the HOMFLY invariant $%
H_{q,t}(\mathcal{L})$ of $\mathcal{L}$. That is
\begin{equation*}
H_{q,t}(\mathcal{L})=[D_{\mathcal{L}}]_{q,t}
\end{equation*}

Note that the corresponding invariant of oriented but unframed links is
\begin{equation*}
<[\mathcal{L}]>=t^{n_{+}-n_{-}}H_{q,t}(\mathcal{L})\text{,}
\end{equation*}

where $n_{+}$ is the number of {%
\includegraphics[height=0.1323in,width=0.2828in]
{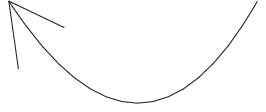}} extrema and $n_{-}$ is the number of ${%
\includegraphics[height=0.1323in,width=0.2231in]
{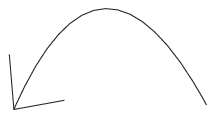}}$ extrema.

The functor $\phi _{q;t,q}$ gives the recursion relation for HOMFLY
polynomials of links. For the invariant of framed links we have:
\begin{equation}
H_{q,tq}(\mathcal{L})=\sum_{D^{\prime }}w(D_{\mathcal{L}},D^{\prime
})_{q,t}[D^{\prime }]_{q,t}
\end{equation}%
Here the summation is taken over all diagrams $D^{\prime }$ which are
obtained from $D_{L}$ by all possible replacements of elementary diagrams
colored by $qt$ by elementary diagrams colored by $t$ and $q$ according to
the way how functor $\phi _{q;t,q}$ acts and $w(D_{\mathcal{L}},D^{\prime
})_{q,t}$ is an integer coefficients polynomial of variables $q$ and $t$.

\section{Recursion corresponding to $sl_{n}\supset sl_{k}\times sl_{n-k}$}

\label{sec3}

In this section we will generalize the results of the previous section by
constructing the functor $\phi _{q;s,u}$ between skein categories ${}%
\mathcal{H}_{q;su}$ and ${}\mathcal{H}_{q;s,u}$. The category ${}\mathcal{H}%
_{q;s,u}$ is defined in the previous section.

Object $(\varepsilon _{1},...,\varepsilon _{n})$ of ${}\mathcal{H}_{q;su}$
is mapped by $\phi _{q;s,u}$ to $\oplus ((\epsilon _{1},s_{1}),\dots
,(\epsilon _{n},s_{n}))$ where the sum is taken over all possible
substitutions $s_{i}=s$ or $s_{i}=u$.

Let $t=su$. Define the functor on elementary diagrams as
\begin{equation}
\phi _{q;s,u}\left(
\raisebox{-0.1609in}
{\includegraphics[height=0.4679in,width=0.5197in]
{positivecross-tt.eps}}\right) =\left(
\begin{array}{cccc}
\raisebox{-0.1609in}
{\includegraphics[height=0.4679in,width=0.5197in]
{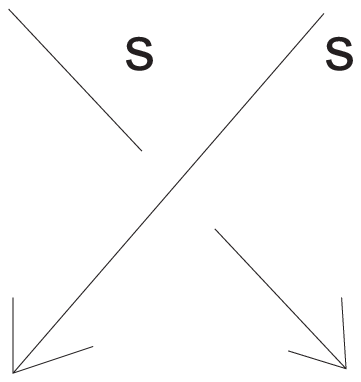}} & 0 & 0 & 0 \\
0 & (q-q^{-1})%
\raisebox{-0.1609in}
{\includegraphics[height=0.5076in,width=0.4947in]
{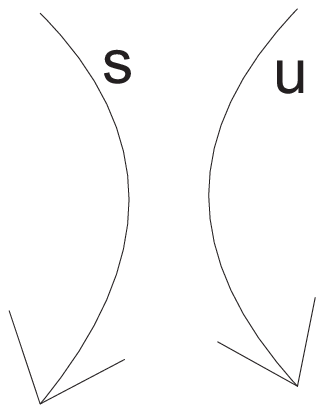}} &
\raisebox{-0.1609in}
{\includegraphics[height=0.4679in,width=0.5197in]
{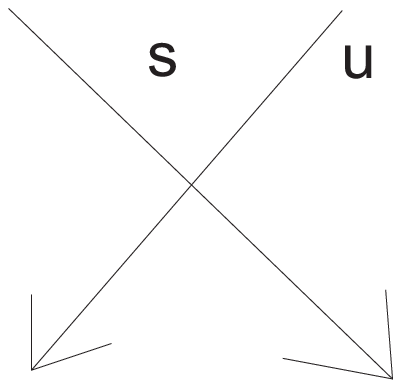}} & 0 \\
&  &  &  \\
0 &
\raisebox{-0.1609in}
{\includegraphics[height=0.4679in,width=0.5197in]
{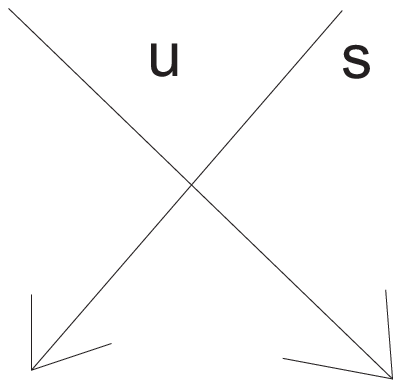}} & 0 & 0 \\
0 & 0 & 0 &
\raisebox{-0.1609in}
{\includegraphics[height=0.4679in,width=0.5197in]
{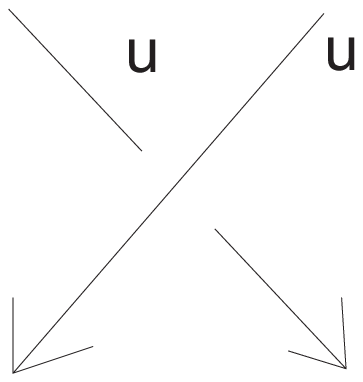}}%
\end{array}%
\right)
\end{equation}

\begin{equation}
\phi _{q;s,u}\left(
\raisebox{-0.1609in}{\includegraphics[height=0.4679in,width=0.5197in]
{negativecross-tt.eps}}\right) =\left(
\begin{array}{cccc}
\raisebox{-0.1609in}
{\includegraphics[height=0.4679in,width=0.5197in]
{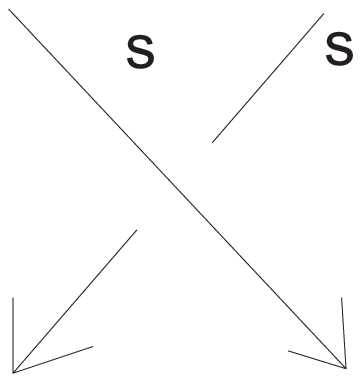}} & 0 & 0 & 0 \\
0 & 0 &
\raisebox{-0.1609in}
{\includegraphics[height=0.4679in,width=0.5197in]
{cross-su.eps}} & 0 \\
&  &  &  \\
0 &
\raisebox{-0.1609in}
{\includegraphics[height=0.4679in,width=0.5197in]
{cross-us.eps}} & -(q-q^{-1})%
\raisebox{-0.1609in}
{\includegraphics[height=0.5076in,width=0.4947in]
{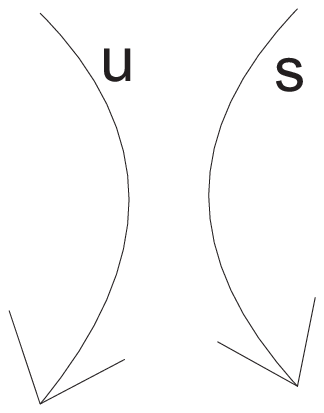}} & 0 \\
0 & 0 & 0 &
\raisebox{-0.1609in}
{\includegraphics[height=0.4679in,width=0.5197in]
{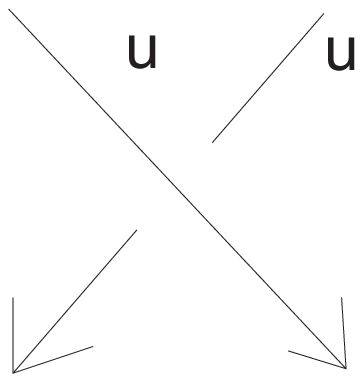}}%
\end{array}%
\right)
\end{equation}

\begin{equation}
\phi _{q;s,u}\left(
\raisebox{-0.1609in}
{\includegraphics[height=0.5076in,width=0.4947in]
{resolvecorss-tt.eps}}\right) =\left(
\begin{array}{cccc}
\raisebox{-0.1609in}
{\includegraphics[height=0.5076in,width=0.4947in]
{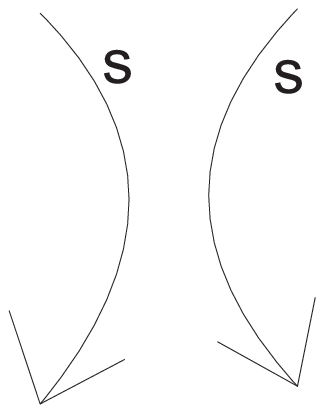}} & 0 & 0 & 0 \\
0 &
\raisebox{-0.1609in}
{\includegraphics[height=0.5076in,width=0.4947in]
{resolvecorss-su.eps}} & 0 & 0 \\
0 & 0 &
\raisebox{-0.1609in}
{\includegraphics[height=0.5076in,width=0.4947in]
{resolvecorss-us.eps}} & 0 \\
0 & 0 & 0 &
\raisebox{-0.1609in}
{\includegraphics[height=0.5076in,width=0.4947in]
{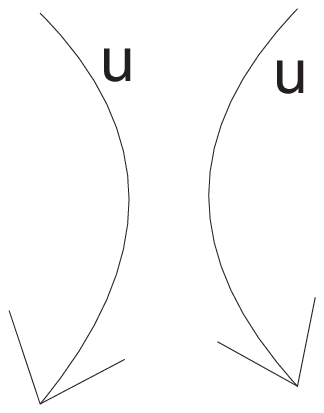}}%
\end{array}%
\right)
\end{equation}

\begin{equation}
\phi _{q;s,u}\left(
\raisebox{-0.1609in}
{\includegraphics[height=0.4765in,width=0.2776in]
{vertical-t.eps}}\right) =\left(
\begin{array}{cc}
\raisebox{-0.1609in}
{\includegraphics[height=0.4756in,width=0.2724in]
{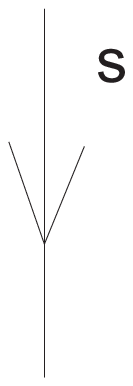}} & 0 \\
0 &
\raisebox{-0.1609in}
{\includegraphics[height=0.4756in,width=0.2724in]
{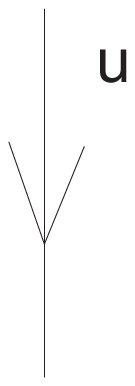}}%
\end{array}%
\right)
\end{equation}%
\begin{equation}
\phi _{q;s,u}\left(
\raisebox{-0.0502in}
{\includegraphics[height=0.1513in,width=0.3442in]
{rightlower-t.eps}}\right) =\left(
\begin{array}{c}
\raisebox{-0.1609in}
{\includegraphics[height=0.1513in,width=0.3442in]
{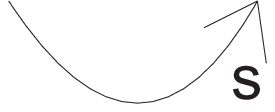}} \\
0 \\
0 \\
\raisebox{-0.1609in}
{\includegraphics[height=0.1513in,width=0.3442in]
{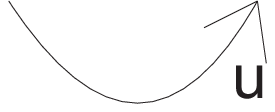}}%
\end{array}%
\right) \text{, }\phi _{q;s,u}\left(
\raisebox{-0.0502in}
{\includegraphics[height=0.1513in,width=0.3312in]
{leftlower-t.eps}}\right) =\left(
\begin{array}{c}
u\text{ }%
\raisebox{-0.0502in}
{\includegraphics[height=0.1513in,width=0.3312in]
{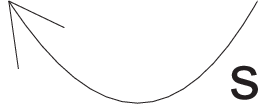}} \\
0 \\
0 \\
s^{-1}\text{ }%
\raisebox{-0.0502in}
{\includegraphics[height=0.1513in,width=0.3312in]
{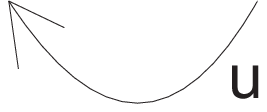}}%
\end{array}%
\right)
\end{equation}%
\begin{equation}
\phi _{q;s,u}\left(
\raisebox{-0.0502in}
{\includegraphics[height=0.1574in,width=0.339in]
{righttop-t.eps}}\right) =\left(
\begin{array}{cccc}
\raisebox{-0.0502in}
{\includegraphics[height=0.1574in,width=0.3528in]
{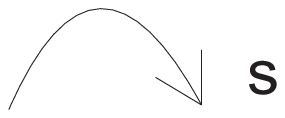}} & 0 & 0 &
\raisebox{-0.0502in}
{\includegraphics[height=0.1574in,width=0.3528in]
{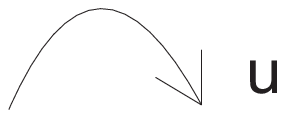}}%
\end{array}%
\right) \text{, }\phi _{q;s,u}\left(
\raisebox{-0.0502in}
{\includegraphics[height=0.1565in,width=0.2984in]
{lefttop-t.eps}}\right) =\left(
\begin{array}{cccc}
u^{-1}\text{ }%
\raisebox{-0.0502in}
{\includegraphics[height=0.1565in,width=0.3148in]
{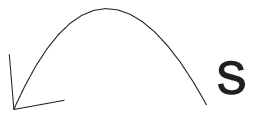}} & 0 & 0 & s\text{ }%
\raisebox{-0.0502in}
{\includegraphics[height=0.1565in,width=0.3148in]
{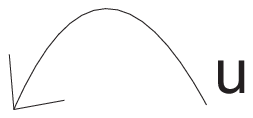}}%
\end{array}%
\right)
\end{equation}

The following statement is a straightforward generalization of the
corresponding theorem from the previous section.

\begin{theorem}
There exists unique covariant braided monoidal functor ${}\mathcal{H}%
_{q;su}\rightarrow \mathcal{H}{}_{q;s,u}$ which acts on elementary diagrams
as above.
\end{theorem}

When $s=q^{k}$ and $u=q^{n-k}$ the functor $\phi _{q;s,u}$ restricts to the
functor between quotient categories and becomes the restriction functor from
the category of $U_{q}(sl_{n})$ modules to the category of $%
U_{q}(sl_{k})\otimes U_{q}(sl_{n-k})$ modules.

\section{Recursion corresponding to $so_{2n}\supset sl_{n}$ and $%
sp_{2n}\supset sl_{n}$}

\label{sec4}

Let $\underline{Tan}^{\prime }$ be the category of non-oriented framed
tangles. Objects in this category are integers, morphisms between objects $%
(n)$ and $(m)$ are non-oriented framed tangles with blackboard framing at
the ends, with $n$ upper ends and with $m$ lower ends. For a ring $A$,
define $\underline{Tan}_{A}^{\prime }$ as the additive $A-$linear category
where objects are direct sum of objects in $\underline{Tan}^{\prime }$ and
morphisms are linear combination of morphisms from $\underline{Tan}^{\prime
} $.

Kauffman invariants are morphisms in the quotient category $\mathcal{K}%
_{q;s} $ of $\underline{Tan}_{{}{}\mathbf{C}[q^{\pm 1},s^{\pm 1}]}^{\prime }$
subject to the following relations

\begin{equation}
\raisebox{-0.2006in}
{\includegraphics[height=0.4687in,width=0.4393in]
{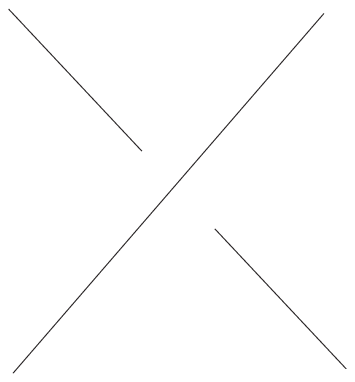}}-%
\raisebox{-0.2006in}
{\includegraphics[height=0.4687in,width=0.4393in]
{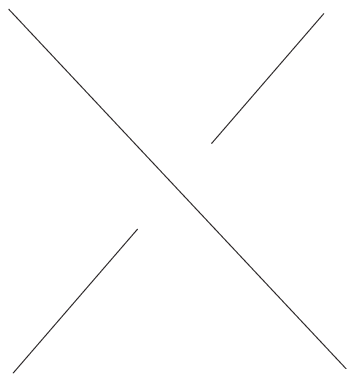}}=(q-q^{-1})\text{ }%
\raisebox{-0.2006in}
{\includegraphics[height=0.5085in,width=0.3286in]
{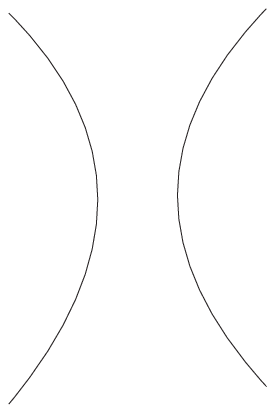}}-(q-q^{-1})\text{ }%
\raisebox{-0.2006in}
{\includegraphics[height=0.4912in,width=0.3485in]
{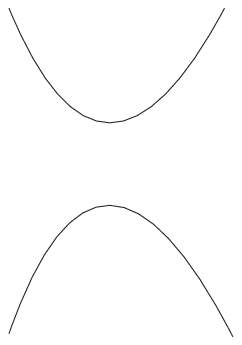}}  \label{Kauffman}
\end{equation}%
and
\begin{equation}
\raisebox{-0.2006in}
{\includegraphics[height=0.5336in,width=0.3503in]
{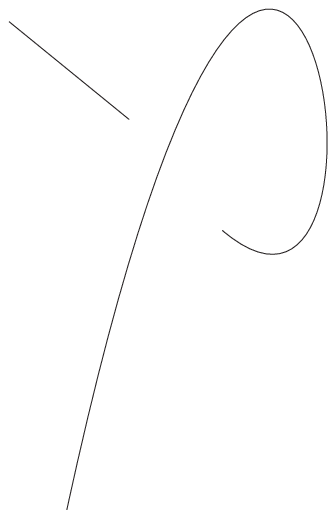}}=s\text{ }%
\raisebox{-0.2006in}
{\includegraphics[height=0.5613in,width=0.147in]
{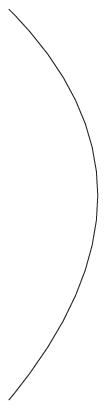}},\text{ \ \ }%
\raisebox{-0.2006in}
{\includegraphics[height=0.5725in,width=0.3563in]
{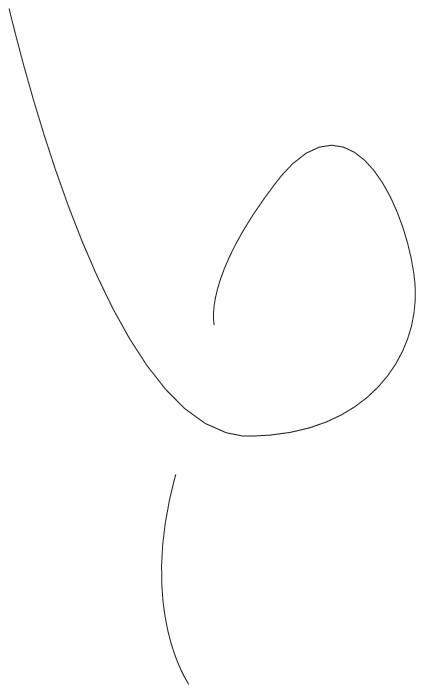}}=s^{-1}\text{ }%
\raisebox{-0.1712in}
{\includegraphics[height=0.5613in,width=0.147in]
{Kresolvekink.eps}}  \label{Kauffman Kink}
\end{equation}

The following theorem describe the covariant functor $\psi _{q;t}$ from $%
\mathcal{K}_{q;t^{2}q^{-1}}$ to $\mathcal{H}_{q;t}$. The category $\mathcal{K%
}_{q;s}$ is a braided monoidal category. Its morphisms are compositions of
tensor products of elementary morphisms \cite{RT}

\begin{theorem}
\label{Thm Kauffman 1} There exists unique functor $\psi _{q;t}$ of braided
monoidal categories $\mathcal{K}_{q;t^{2}q^{-1}}\rightarrow \mathcal{H}%
_{q;t} $ which acts on objects as $\psi _{q;t}(n)=\underset{\varepsilon
_{1},...,\varepsilon _{n}=\pm }{\oplus }(\varepsilon _{1},...,\varepsilon
_{n})$ with its action on elementary morphisms described below
\begin{equation}
\psi _{q;t}\left(
\raisebox{-0.2006in}
{\includegraphics[height=0.4687in,width=0.4393in]
{Kpositivecross.eps}}\right) =\left(
\begin{array}{cccc}
\raisebox{-0.2006in}
{\includegraphics[height=0.4687in,width=0.4393in]
{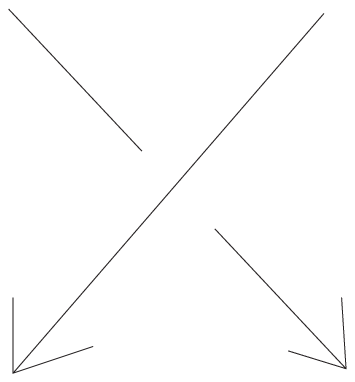}} & 0 & 0 & 0 \\
0 & a%
\raisebox{-0.2006in}
{\includegraphics[height=0.5076in,width=0.3779in]
{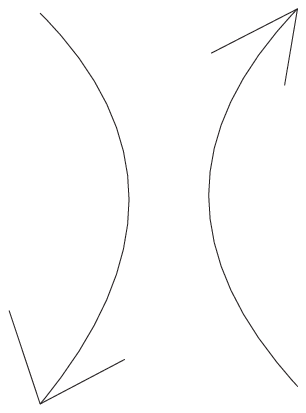}}+b%
\raisebox{-0.2006in}
{\includegraphics[height=0.4912in,width=0.3494in]
{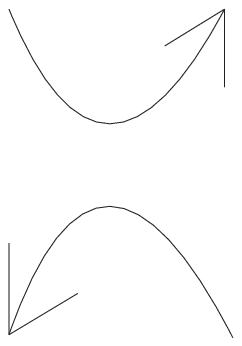}} &
\raisebox{-0.2006in}
{\includegraphics[height=0.4687in,width=0.4393in]
{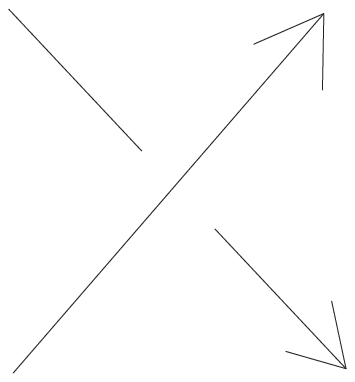}} & 0 \\
&  &  &  \\
0 &
\raisebox{-0.2006in}
{\includegraphics[height=0.4687in,width=0.4393in]
{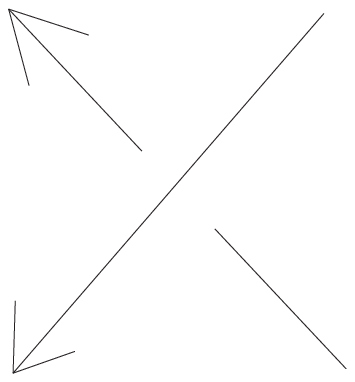}} & 0 & 0 \\
0 & 0 & 0 &
\raisebox{-0.2006in}
{\includegraphics[height=0.4687in,width=0.4393in]
{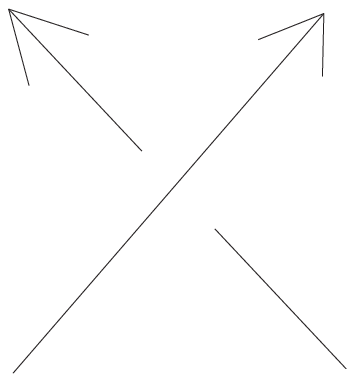}}%
\end{array}%
\right)
\end{equation}
\end{theorem}

\begin{equation}
\psi _{q;t}\left(
\raisebox{-0.2006in}
{\includegraphics[height=0.4687in,width=0.4393in]
{Knegativecross.eps}}\right) =\left(
\begin{array}{cccc}
\raisebox{-0.2006in}
{\includegraphics[height=0.4687in,width=0.4393in]
{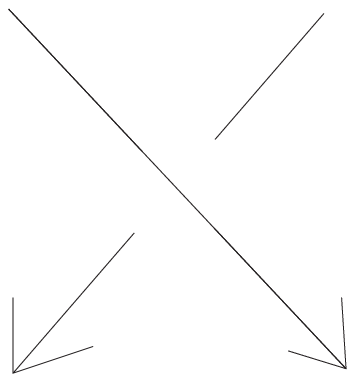}} & 0 & 0 & 0 \\
0 & 0 &
\raisebox{-0.2006in}
{\includegraphics[height=0.4687in,width=0.4393in]
{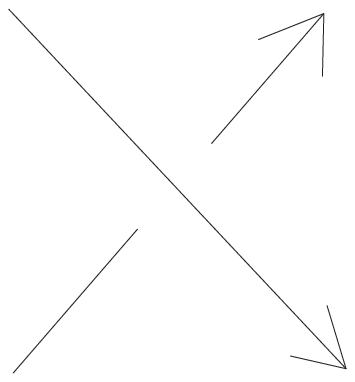}} & 0 \\
&  &  &  \\
0 &
\raisebox{-0.2006in}
{\includegraphics[height=0.4687in,width=0.4393in]
{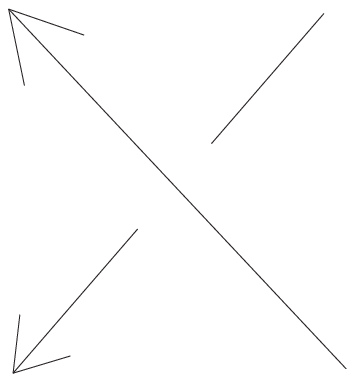}} & c%
\raisebox{-0.2006in}
{\includegraphics[height=0.5076in,width=0.3615in]
{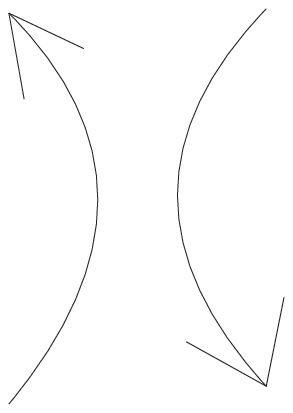}}+d%
\raisebox{-0.2006in}
{\includegraphics[height=0.4903in,width=0.3771in]
{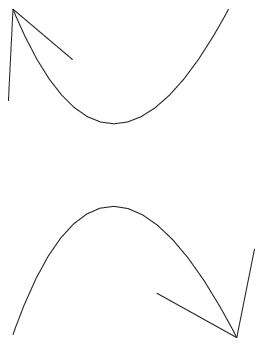}} & 0 \\
0 & 0 & 0 &
\raisebox{-0.2006in}
{\includegraphics[height=0.4687in,width=0.4393in]
{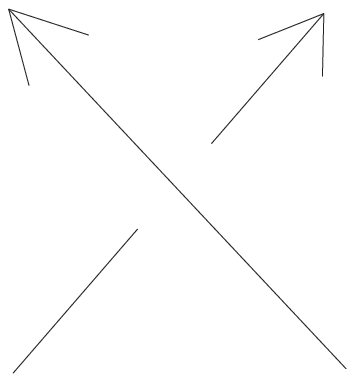}}%
\end{array}%
\right)
\end{equation}%
\begin{equation}
\psi _{q;t}\left(
\raisebox{-0.2006in}
{\includegraphics[height=0.5085in,width=0.339in]
{Kresolvecorss.eps}}\right) =\left(
\begin{array}{cccc}
\raisebox{-0.2006in}
{\includegraphics[height=0.5085in,width=0.3995in]
{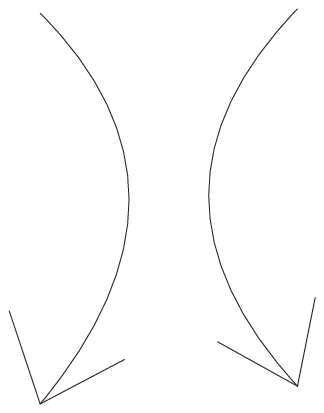}} & 0 & 0 & 0 \\
0 &
\raisebox{-0.2006in}
{\includegraphics[height=0.5076in,width=0.3779in]
{lower-upper-resolvecorss.eps}} & 0 & 0 \\
0 & 0 &
\raisebox{-0.2006in}
{\includegraphics[height=0.5076in,width=0.3615in]
{upper-lower-resolvecorss.eps}} & 0 \\
0 & 0 & 0 &
\raisebox{-0.2006in}
{\includegraphics[height=0.5085in,width=0.339in]
{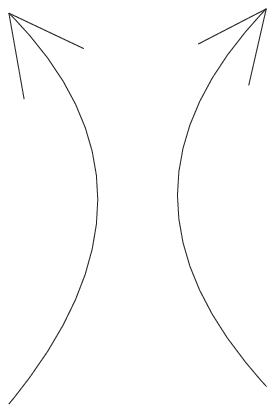}}%
\end{array}%
\right)
\end{equation}

\begin{equation}
\psi _{q;t}\left(
\raisebox{-0.0502in}
{\includegraphics[height=0.1513in,width=0.3312in]
{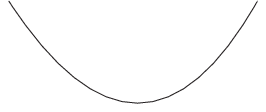}}\right) =\left(
\begin{array}{c}
0 \\
\raisebox{-0.0502in}
{\includegraphics[height=0.1513in,width=0.3442in]
{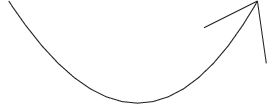}} \\
tq^{-1}\text{ }%
\raisebox{-0.0502in}
{\includegraphics[height=0.1513in,width=0.3312in]
{leftlower.eps}} \\
0%
\end{array}%
\right) \text{, }\psi _{q;t}\left(
\raisebox{-0.0502in}
{\includegraphics[height=0.1574in,width=0.2646in]
{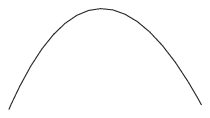}}\right) =\left(
\begin{array}{cccc}
0 & t^{-1}q\text{ }%
\raisebox{-0.0502in}
{\includegraphics[height=0.1565in,width=0.2681in]
{lefttop.eps}} &
\raisebox{-0.0502in}
{\includegraphics[height=0.1574in,width=0.2646in]
{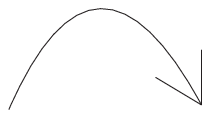}} & 0%
\end{array}%
\right)
\end{equation}

Here coefficients $a$, $b$, $c$ and $d$ are

\begin{equation}
a=q-q^{-1},b=-(q-q^{-1})t^{-1}q,c=-(q-q^{-1})\text{ and }d=(q-q^{-1})tq^{-1}
\end{equation}

\begin{proof}
To prove the theorem we should check skein relation. Here, as in section $1$
we assume that the functor $\psi _{q;t}$ bring the left side of equalities
of diagrams below to the right side. Now, let us check skein relations. The
image of the left side of (\ref{Kauffman}) in $\mathcal{H}_{q;t}$ is

$\left(
\begin{array}{cccc}
\raisebox{-0.1003in}
{\includegraphics[height=0.3217in,width=0.3027in]
{positivecross.eps}}-%
\raisebox{-0.1003in}
{\includegraphics[height=0.3217in,width=0.3027in]
{negativecross.eps}} & 0 & 0 & 0 \\
0 & a%
\raisebox{-0.1003in}
{\includegraphics[height=0.3494in,width=0.2612in]
{lower-upper-resolvecorss.eps}}+b%
\raisebox{-0.1003in}
{\includegraphics[height=0.3598in,width=0.2569in]
{right-left-vertical-resolvecorss.eps}} &
\raisebox{-0.1003in}
{\includegraphics[height=0.3217in,width=0.3027in]
{right-negativecross.eps}}-%
\raisebox{-0.1003in}
{\includegraphics[height=0.3217in,width=0.3027in]
{right-positivecross.eps}} & 0 \\
&  &  &  \\
0 &
\raisebox{-0.1003in}
{\includegraphics[height=0.3217in,width=0.3027in]
{left-negativecross.eps}}-%
\raisebox{-0.1003in}
{\includegraphics[height=0.3217in,width=0.3027in]
{left-positivecross.eps}} & -c%
\raisebox{-0.1003in}
{\includegraphics[height=0.3494in,width=0.2499in]
{upper-lower-resolvecorss.eps}}-d%
\raisebox{-0.1003in}
{\includegraphics[height=0.3598in,width=0.2767in]
{left-right-vertical-resolvecorss.eps}} & 0 \\
0 & 0 & 0 &
\raisebox{-0.1003in}
{\includegraphics[height=0.3217in,width=0.3027in]
{upper-positivecross.eps}}-%
\raisebox{-0.1003in}
{\includegraphics[height=0.3217in,width=0.3027in]
{upper-negativecross.eps}}%
\end{array}%
\right) $

$=\left(
\begin{array}{cccc}
(q-q^{-1})%
\raisebox{-0.1003in}
{\includegraphics[height=0.3494in,width=0.275in]
{resolvecorss.eps}} & 0 & 0 & 0 \\
0 & a%
\raisebox{-0.1003in}
{\includegraphics[height=0.3494in,width=0.2612in]
{lower-upper-resolvecorss.eps}}+b%
\raisebox{-0.1003in}
{\includegraphics[height=0.3598in,width=0.2569in]
{right-left-vertical-resolvecorss.eps}} &
\raisebox{-0.1003in}
{\includegraphics[height=0.3217in,width=0.3027in]
{right-negativecross.eps}}-%
\raisebox{-0.1003in}
{\includegraphics[height=0.3217in,width=0.3027in]
{right-positivecross.eps}} & 0 \\
&  &  &  \\
0 &
\raisebox{-0.1003in}
{\includegraphics[height=0.3217in,width=0.3027in]
{left-negativecross.eps}}-%
\raisebox{-0.1003in}
{\includegraphics[height=0.3217in,width=0.3027in]
{left-positivecross.eps}} & -c%
\raisebox{-0.1003in}
{\includegraphics[height=0.3494in,width=0.2499in]
{upper-lower-resolvecorss.eps}}-d%
\raisebox{-0.1003in}
{\includegraphics[height=0.3598in,width=0.2767in]
{left-right-vertical-resolvecorss.eps}} & 0 \\
0 & 0 & 0 & (q-q^{-1})%
\raisebox{-0.1003in}
{\includegraphics[height=0.3485in,width=0.2352in]
{upper-uper-resolvecorss.eps}}%
\end{array}%
\right) $

The functor $\psi _{q;t}$ maps the right hand side of (\ref{Kauffman}) to

$(q-q^{-1})\left(
\begin{array}{cccc}
\raisebox{-0.1003in}
{\includegraphics[height=0.3494in,width=0.275in]
{resolvecorss.eps}} & 0 & 0 & 0 \\
0 &
\raisebox{-0.1003in}
{\includegraphics[height=0.3494in,width=0.2612in]
{lower-upper-resolvecorss.eps}}-t^{-1}q%
\raisebox{-0.1003in}
{\includegraphics[height=0.3598in,width=0.2569in]
{right-left-vertical-resolvecorss.eps}} & -%
\raisebox{-0.1003in}
{\includegraphics[height=0.3312in,width=0.2534in]
{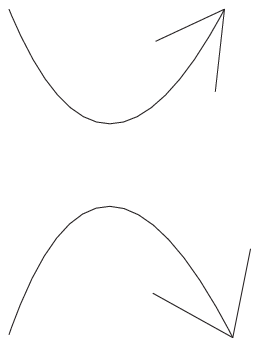}} & 0 \\
&  &  &  \\
0 & -%
\raisebox{-0.1003in}
{\includegraphics[height=0.3329in,width=0.2421in]
{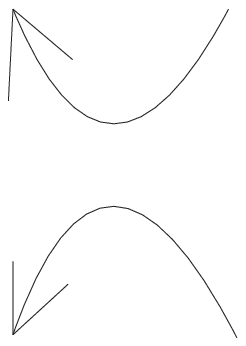}} &
\raisebox{-0.1003in}
{\includegraphics[height=0.3503in,width=0.2499in]
{upper-lower-resolvecorss.eps}}-tq^{-1}%
\raisebox{-0.1003in}
{\includegraphics[height=0.3312in,width=0.2569in]
{left-right-vertical-resolvecorss.eps}} & 0 \\
0 & 0 & 0 &
\raisebox{-0.1003in}
{\includegraphics[height=0.3485in,width=0.2352in]
{upper-uper-resolvecorss.eps}}%
\end{array}%
\right) $

Comparing the coefficients of these diagrams for values of $a,b,c$ and $d$
given in the theorem, we have the desired identity of $\psi _{q;t}$.

Framing moves can be checked similarly. For example, for one of the
nontrivial components of the image of the left move in (\ref{Kauffman Kink})
we have:

$%
\raisebox{-0.1003in}
{\includegraphics[height=0.3217in,width=0.3027in]
{positivecross.eps}}_{%
\raisebox{-0.1203in}
{\includegraphics[height=0.1513in,width=0.3442in]
{rightlower.eps}}}^{%
\raisebox{0.0803in}
{\includegraphics[height=0.1565in,width=0.2681in]
{lefttop.eps}}}$ $t^{-1}q+(a$ $%
\raisebox{-0.1003in}
{\includegraphics[height=0.3494in,width=0.2612in]
{lower-upper-resolvecorss.eps}}+b$ $%
\raisebox{-0.1003in}
{\includegraphics[height=0.3598in,width=0.2569in]
{right-left-vertical-resolvecorss.eps}})tq^{-1}%
\begin{array}{c}
\raisebox{0.0803in}
{\includegraphics[height=0.1565in,width=0.2681in]
{righttop.eps}} \\
\raisebox{-0.1203in}
{\includegraphics[height=0.1513in,width=0.3442in]
{leftlower.eps}}%
\end{array}%
=%
\raisebox{-0.2006in}
{\includegraphics[height=0.5336in,width=0.3511in]
{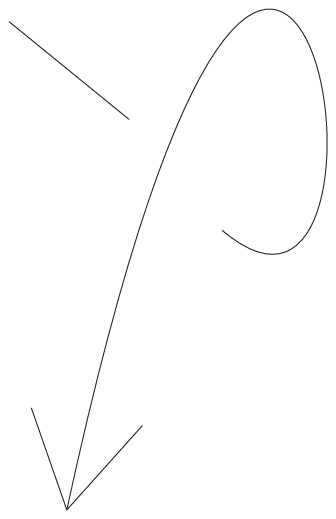}}$ $t^{-1}q+a$ $%
\raisebox{-0.2006in}
{\includegraphics[height=0.5007in,width=0.6737in]
{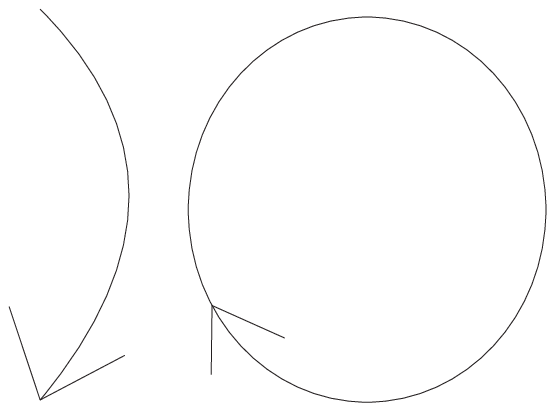}}$ $tq^{-1}+b$ $%
\raisebox{-0.2006in}
{\includegraphics[height=0.5587in,width=0.1894in]
{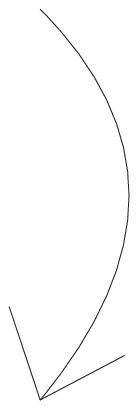}}$ $tq^{-1}$

$=%
\raisebox{-0.2006in}
{\includegraphics[height=0.5587in,width=0.1894in]
{resolvekink.eps}}$ $q+%
\raisebox{-0.2006in}
{\includegraphics[height=0.5587in,width=0.1894in]
{resolvekink.eps}}$ $a\frac{t-t^{-1}}{q-q^{-1}}tq^{-1}+b$ $%
\raisebox{-0.2006in}
{\includegraphics[height=0.5587in,width=0.1894in]
{resolvekink.eps}}$ $tq^{-1}=%
\raisebox{-0.2006in}
{\includegraphics[height=0.5587in,width=0.1894in]
{resolvekink.eps}}$ $%
[q+(t-t^{-1})tq^{-1}-(q-q^{-1})t^{-1}qtq^{-1}]=t^{2}q^{-1}$ $%
\raisebox{-0.2006in}
{\includegraphics[height=0.5587in,width=0.1894in]
{resolvekink.eps}}$

$=s$ $%
\raisebox{-0.2006in}
{\includegraphics[height=0.5587in,width=0.1894in]
{resolvekink.eps}}$

and

$%
\raisebox{-0.0796in}
{\includegraphics[height=0.2032in,width=0.1929in]
{upper-positivecross.eps}}_{%
\raisebox{-0.1203in}
{\includegraphics[height=0.1513in,width=0.3442in]
{leftlower.eps}}}^{%
\raisebox{0.0803in}
{\includegraphics[height=0.1565in,width=0.2681in]
{righttop.eps}}}$ $tq^{-1}=t$ $%
\raisebox{-0.2214in}
{\includegraphics[height=0.6244in,width=0.2084in]
{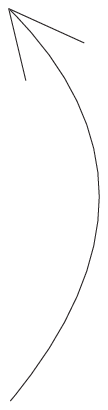}}$ $tq^{-1}=s$ $%
\raisebox{-0.2214in}
{\includegraphics[height=0.6244in,width=0.2084in]
{resolvekinkreverse.eps}}$

It is easy to check that other relations hold as well, in particular we have

$\psi _{q;t}\left(
\raisebox{-0.2006in}
{\includegraphics[height=0.4955in,width=0.3269in] {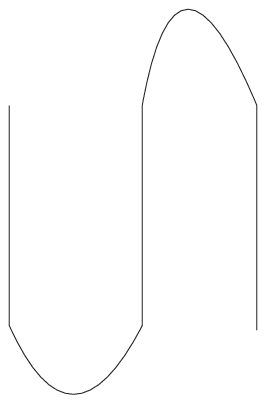}}\right) $ $%
=\psi _{q;t}\left(
\raisebox{-0.2006in}
{\includegraphics[trim=0.051881in 0.000000in 0.051881in
0.000000in, height=0.4601in,width=0.1193in]
{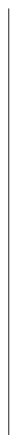}}\right) $, $\psi _{q;t}\left(
\raisebox{-0.1003in}
{\includegraphics[height=0.3226in,width=0.3027in]
{Knegativecross.eps}}\right) =\psi _{q;t}\left(
\raisebox{-0.1003in}
{\includegraphics[height=0.3226in,width=0.5027in]
{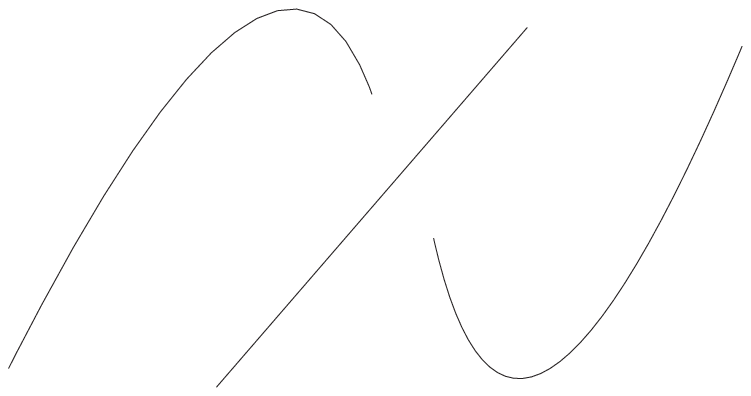}}\right) $

To extend the map $\psi _{q;t}$ from elementary diagrams to any tangle, one
should check relations between elementary diagrams. This is completely
parallel to \cite{T}. For example, here is the proof of invertibility of the
braiding

$\psi _{q;t}\left(
\begin{array}{c}
\raisebox{-0.1003in}
{\includegraphics[height=0.3217in,width=0.3027in]
{Kpositivecross.eps}} \\
\raisebox{-0.1003in}
{\includegraphics[height=0.3217in,width=0.3027in]
{Knegativecross.eps}}%
\end{array}%
\right) =\left(
\begin{array}{cccc}
\begin{array}{c}
\raisebox{-0.1003in}
{\includegraphics[height=0.3217in,width=0.3027in]
{positivecross.eps}} \\
\raisebox{-0.1003in}
{\includegraphics[height=0.3217in,width=0.3027in]
{negativecross.eps}}%
\end{array}
& 0 & 0 & 0 \\
0 &
\begin{array}{c}
\raisebox{-0.1003in}
{\includegraphics[height=0.3217in,width=0.3027in]
{right-negativecross.eps}} \\
\raisebox{-0.1003in}
{\includegraphics[height=0.3217in,width=0.3027in]
{left-positivecross.eps}}%
\end{array}
& \left( a%
\begin{array}{c}
\raisebox{-0.1003in}{\includegraphics[height=0.3494in,width=0.2612in]
{lower-upper-resolvecorss.eps}} \\
\raisebox{-0.1003in}{\includegraphics[height=0.3217in,width=0.3027in]
{right-positivecross.eps}}%
\end{array}%
+b%
\begin{array}{c}
\raisebox{-0.1003in}
{\includegraphics[height=0.3598in,width=0.2569in]
{right-left-vertical-resolvecorss.eps}} \\
\raisebox{-0.1003in}
{\includegraphics[height=0.3217in,width=0.3027in]
{right-positivecross.eps}}%
\end{array}%
+c%
\begin{array}{c}
\raisebox{-0.1003in}
{\includegraphics[height=0.3217in,width=0.3027in]
{right-negativecross.eps}} \\
\raisebox{-0.1003in}
{\includegraphics[height=0.3494in,width=0.2499in]
{upper-lower-resolvecorss.eps}}%
\end{array}%
+d%
\begin{array}{c}
\raisebox{-0.1003in}
{\includegraphics[height=0.3217in,width=0.3027in]
{right-negativecross.eps}} \\
\raisebox{-0.1003in}
{\includegraphics[height=0.3312in,width=0.2569in]
{left-right-vertical-resolvecorss.eps}}%
\end{array}%
\right) & 0 \\
&  &  &  \\
0 & 0 &
\begin{array}{c}
\raisebox{-0.1003in}
{\includegraphics[height=0.3217in,width=0.3027in]
{left-negativecross.eps}} \\
\raisebox{-0.1003in}
{\includegraphics[height=0.3217in,width=0.3027in]
{right-positivecross.eps}}%
\end{array}
& 0 \\
0 & 0 & 0 &
\begin{array}{c}
\raisebox{-0.1003in}
{\includegraphics[height=0.3217in,width=0.3027in]
{upper-positivecross.eps}} \\
\raisebox{-0.1003in}
{\includegraphics[height=0.3217in,width=0.3027in]
{upper-negativecross.eps}}%
\end{array}%
\end{array}%
\right) $

$=\left(
\begin{array}{cccc}
\raisebox{-0.1003in}
{\includegraphics[height=0.3494in,width=0.275in]
{resolvecorss.eps}} & 0 & 0 & 0 \\
0 &
\raisebox{-0.1003in}
{\includegraphics[height=0.3494in,width=0.2612in]
{lower-upper-resolvecorss.eps}} & \left( a\text{ }%
\raisebox{-0.1003in}
{\includegraphics[height=0.3217in,width=0.3027in]
{right-positivecross.eps}}+bt\text{ }%
\raisebox{-0.1003in}
{\includegraphics[height=0.3312in,width=0.2534in]
{right-right-vertical-resolvecorss.eps}}+c\text{ }%
\raisebox{-0.1003in}
{\includegraphics[height=0.3217in,width=0.3027in]
{right-negativecross.eps}}+dt^{-1}\text{ }%
\raisebox{-0.1003in}
{\includegraphics[height=0.3312in,width=0.2534in]
{right-right-vertical-resolvecorss.eps}}\right) & 0 \\
&  &  &  \\
0 & 0 &
\raisebox{-0.1003in}
{\includegraphics[height=0.3494in,width=0.2499in]
{upper-lower-resolvecorss.eps}} & 0 \\
0 & 0 & 0 &
\raisebox{-0.1003in}
{\includegraphics[height=0.3485in,width=0.2352in]
{upper-uper-resolvecorss.eps}}%
\end{array}%
\right) $

$=\left(
\begin{array}{cccc}
\raisebox{-0.1003in}
{\includegraphics[height=0.3494in,width=0.275in]
{resolvecorss.eps}} & 0 & 0 & 0 \\
0 &
\raisebox{-0.1003in}
{\includegraphics[height=0.3494in,width=0.2612in]
{lower-upper-resolvecorss.eps}} & [a(q-q^{-1})+(bt+dt^{-1})]\text{ }%
\raisebox{-0.1003in}
{\includegraphics[height=0.3312in,width=0.2534in]
{right-right-vertical-resolvecorss.eps}} & 0 \\
&  &  &  \\
0 & 0 &
\raisebox{-0.1003in}
{\includegraphics[height=0.3494in,width=0.2499in]
{upper-lower-resolvecorss.eps}} & 0 \\
0 & 0 & 0 &
\raisebox{-0.1003in}
{\includegraphics[height=0.3485in,width=0.2352in]
{upper-uper-resolvecorss.eps}}%
\end{array}%
\right) $

$=\left(
\begin{array}{cccc}
\raisebox{-0.2006in}
{\includegraphics[height=0.5085in,width=0.3995in]
{resolvecorss.eps}} & 0 & 0 & 0 \\
0 &
\raisebox{-0.2006in}
{\includegraphics[height=0.5076in,width=0.3779in]
{lower-upper-resolvecorss.eps}} & 0 & 0 \\
0 & 0 &
\raisebox{-0.2006in}
{\includegraphics[height=0.5076in,width=0.3615in]
{upper-lower-resolvecorss.eps}} & 0 \\
0 & 0 & 0 &
\raisebox{-0.2006in}
{\includegraphics[height=0.5085in,width=0.339in]
{upper-uper-resolvecorss.eps}}%
\end{array}%
\right) =\psi _{q;t}\left(
\raisebox{-0.2006in}
{\includegraphics[height=0.5085in,width=0.339in]
{Kresolvecorss.eps}}\right) $

Other relations can be checked similarly.
\end{proof}

When $s=\pm q^{2n-1}$ and $t=q^{n}$, the functor $\psi _{q;t}$ restricts to
the functor between corresponding quotient categories and becomes the
restriction functor from $U_{q}(so_{2n})$ or $U_{q}(sp_{2n})$ modules to $%
U_{q}(sl_{n})$ modules.

\section{The Recursion corresponding to $so_{2n}\supset so_{2k}\times
sl_{n-k}$, $so_{2n+1}\supset so_{2k+1}\times sl_{n-k}$ and $sp_{2n}\supset
sp_{2k}\times sl_{n-k}$}

\label{sec5}

Let $g_{n}$ be either $so_{2n+1}$, $so_{2n}$ or $sp_{2n}$. In this section
we will construct the functor between skein categories corresponding to
embeddings $g_{n}\supset g_{k}\times sl_{n-k}$ of Lie algebras.

The skein category $\mathcal{K}_{q;s}$ corresponds to $g_{k}$ and to
Kauffman invariants we described in the previous section. Now we will define
the skein category $\mathcal{HK}_{q;s,t}$ corresponding to products of
HOMFLY and Kauffman invariants and will describe the functor

\begin{equation}
\chi _{q;s,t}:\mathcal{K}_{q;st^{2}}\rightarrow \mathcal{HK}_{q;s,t}
\end{equation}

First, consider the category $\underline{Tan}^{\prime \prime }$ of framed
tangles with some components being oriented and some not.

Objects of $\underline{Tan}^{\prime \prime }$ are sequences of $%
\{\varepsilon _{1},...,\varepsilon _{n}\}$ where $\varepsilon _{i}=\pm ,0$.
Morphisms between $\{\varepsilon \}$ and $\{\sigma \}$ are framed tangles
with the blackboard framing near ends. Some components of these tangles are
oriented, some not. The orientation of components agrees with the objects as
it is shown on the following figure.

$\raisebox{0in}
{\includegraphics[height=2.28in,width=2.08in]
{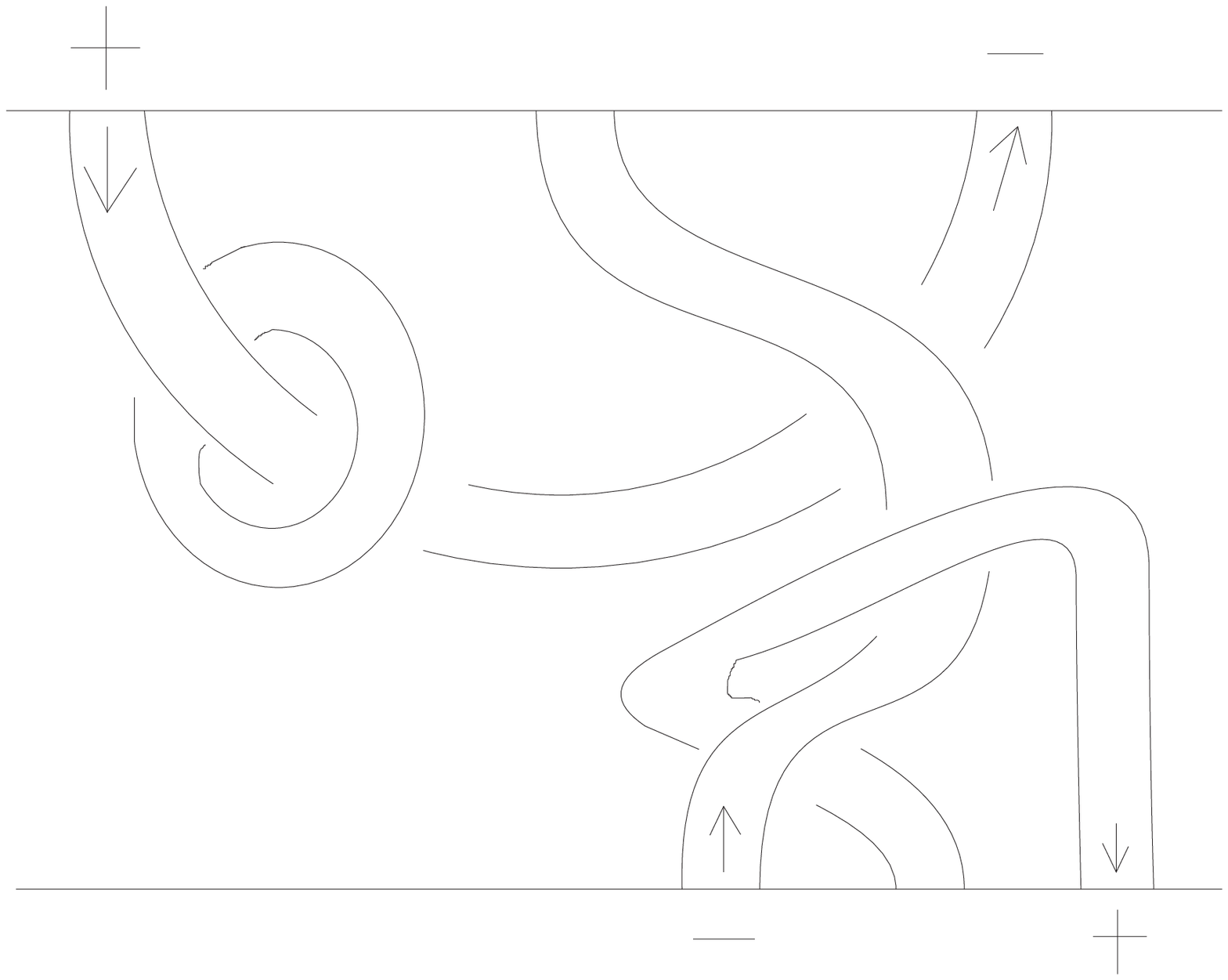}}$

Objects of the category $\mathcal{HK}_{q;s,t}$ are direct sums of objects in
$\underline{Tan}^{\prime \prime }$. Morphisms between $\{\varepsilon \}$ and
$\{\sigma \}$ are quotient spaces of linear combination of morphisms in $%
\underline{Tan}^{\prime }$ modulo defining relations as in $\mathcal{K}%
_{q;s} $ for non-oriented components, as in $\mathcal{H}_{q;t}$ for oriented
components and with the following relations between oriented and
non-oriented components

\bigskip

$%
\raisebox{-0.2006in}
{\includegraphics[height=0.4687in,width=0.4393in]
{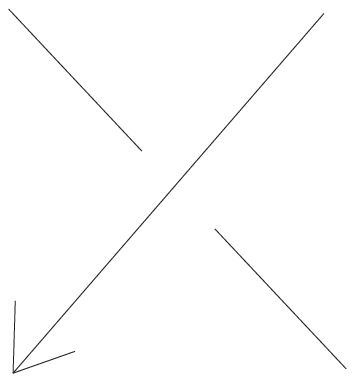}}$ $=$ $%
\raisebox{-0.2006in}
{\includegraphics[height=0.4687in,width=0.4393in]
{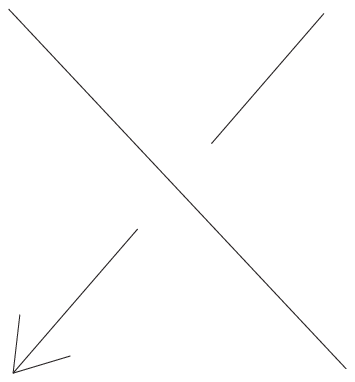}}$ $=$ $%
\raisebox{-0.2006in}
{\includegraphics[height=0.4687in,width=0.4393in]
{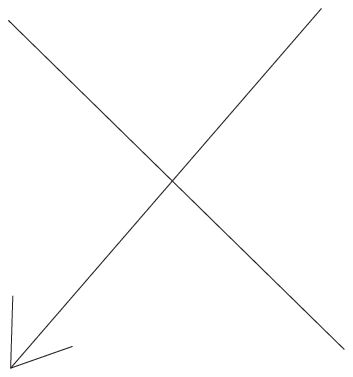}}$ , $%
\raisebox{-0.2006in}
{\includegraphics[height=0.4687in,width=0.4393in]
{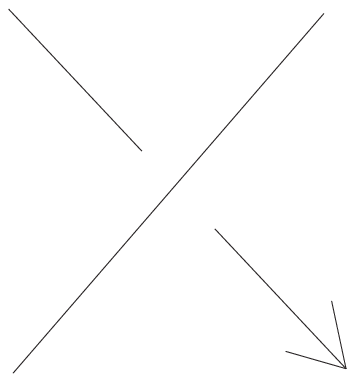}}$ $=$ $%
\raisebox{-0.2006in}
{\includegraphics[height=0.4687in,width=0.4393in]
{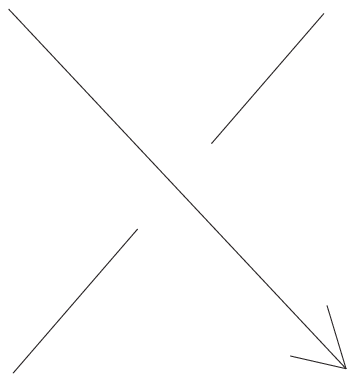}}$ $=$ $%
\raisebox{-0.2006in}
{\includegraphics[height=0.4687in,width=0.4393in]
{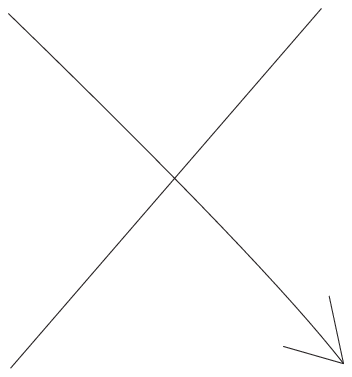}}$,

$\bigskip $

$%
\raisebox{-0.2006in}
{\includegraphics[height=0.4687in,width=0.4393in]
{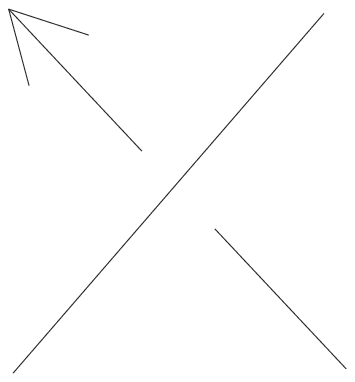}}$ $=$ $%
\raisebox{-0.2006in}
{\includegraphics[height=0.4687in,width=0.4393in]
{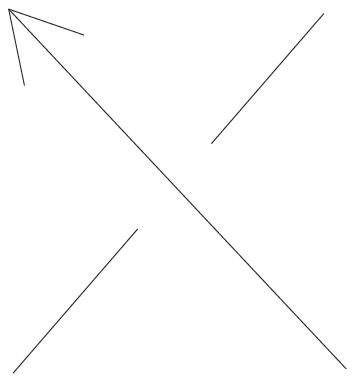}}$ $=$ $%
\raisebox{-0.2006in}
{\includegraphics[height=0.4687in,width=0.4393in]
{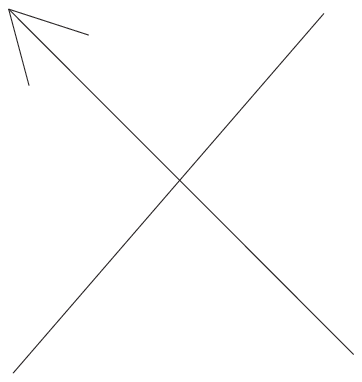}}$ , $%
\raisebox{-0.2006in}
{\includegraphics[height=0.4687in,width=0.4393in]
{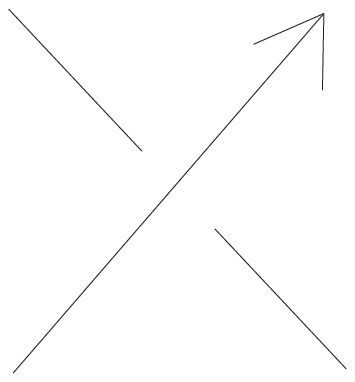}}$ $=$ $%
\raisebox{-0.2006in}
{\includegraphics[height=0.4687in,width=0.4393in]
{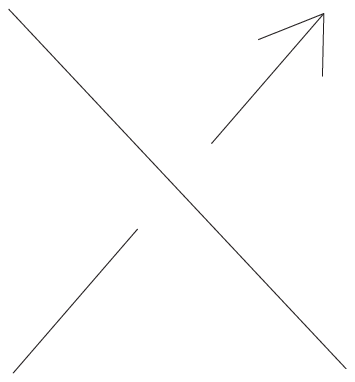}}$ $=$ $%
\raisebox{-0.2006in}
{\includegraphics[height=0.4687in,width=0.4393in]
{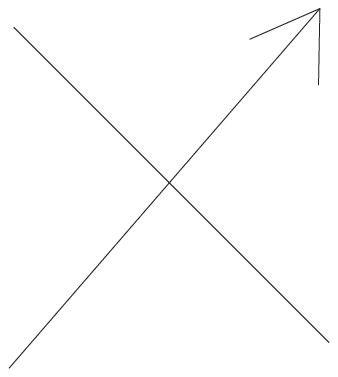}}$.

\bigskip

Now we will construct a covariant functor $\chi _{q;s,t}:\mathcal{K}%
_{q;st^{2}}\rightarrow \mathcal{HK}_{q;s,t}$. Define it on objects as $\chi
_{q;s,t}(n)=\underset{\varepsilon _{1},...,\varepsilon _{n}=\pm ,0}{\oplus }%
(\varepsilon _{1},...,\varepsilon _{n})$. To define the functor $\chi
_{q;s,t}$ on morphisms it is convenient to introduce a formal variable $a$
such that $s=a^{2}q^{-1}$ and $p=q-q^{-1}$.

Then we define $\chi _{q;s,t}$ on elementary morphism as follows.

\bigskip

$\chi _{q;s,t}\left(
\raisebox{-0.0796in}
{\includegraphics[height=0.2491in,width=0.2335in]
{Kpositivecross.eps}}\right) $

\bigskip

$=\left(
\begin{array}{ccccccccc}
\raisebox{-0.0796in}
{\includegraphics[height=0.2032in,width=0.1929in]
{positivecross.eps}} & 0 & 0 & 0 & 0 & 0 & 0 & 0 & 0 \\
0 & p%
\raisebox{-0.0796in}
{\includegraphics[height=0.2205in,width=0.1678in]
{lower-upper-resolvecorss.eps}}-ps^{-1}t^{-1}%
\raisebox{-0.0796in}
{\includegraphics[height=0.2266in,width=0.1652in]
{right-left-vertical-resolvecorss.eps}} &
\raisebox{-0.0796in}
{\includegraphics[height=0.2032in,width=0.1929in]
{right-negativecross.eps}} & 0 & 0 & 0 & 0 & 0 & -pa^{-1}%
\raisebox{-0.0796in}
{\includegraphics[height=0.2266in,width=0.1652in]
{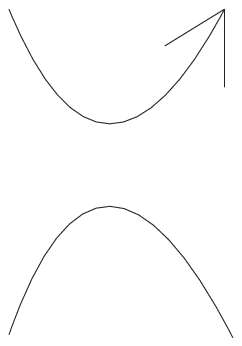}} \\
0 &
\raisebox{-0.0796in}
{\includegraphics[height=0.2032in,width=0.1929in]
{left-negativecross.eps}} & 0 & 0 & 0 & 0 & 0 & 0 & 0 \\
0 & 0 & 0 &
\raisebox{-0.0796in}
{\includegraphics[height=0.2032in,width=0.1929in]
{upper-positivecross.eps}} & 0 & 0 & 0 & 0 & 0 \\
0 & 0 & 0 & 0 & p%
\raisebox{-0.0796in}
{\includegraphics[height=0.2205in,width=0.1678in]
{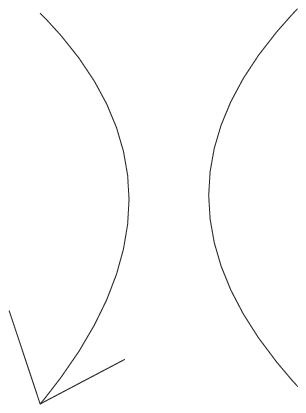}} & 0 &
\raisebox{-0.0796in}
{\includegraphics[height=0.2032in,width=0.1929in]
{rightlower-cross.eps}} & 0 & 0 \\
0 & 0 & 0 & 0 & 0 & p%
\raisebox{-0.0796in}
{\includegraphics[height=0.2205in,width=0.1531in]
{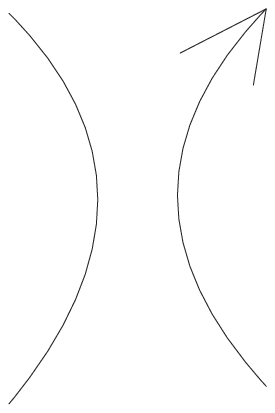}} & 0 &
\raisebox{-0.0796in}
{\includegraphics[height=0.2032in,width=0.1929in]
{rightupper-cross.eps}} & 0 \\
0 & 0 & 0 & 0 &
\raisebox{-0.0796in}
{\includegraphics[height=0.2032in,width=0.1929in]
{leftlower-cross.eps}} & 0 & 0 & 0 & 0 \\
0 & 0 & 0 & 0 & 0 &
\raisebox{-0.0796in}
{\includegraphics[height=0.2032in,width=0.1929in]
{leftupper-cross.eps}} & 0 & 0 & 0 \\
0 & -pt^{-1}a^{-1}q%
\raisebox{-0.0796in}
{\includegraphics[height=0.2266in,width=0.1652in]
{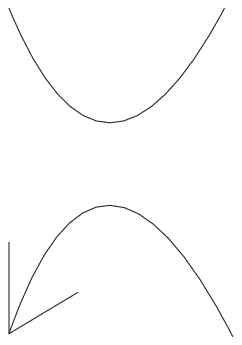}} & 0 & 0 & 0 & 0 & 0 & 0 &
\raisebox{-0.0796in}
{\includegraphics[height=0.2032in,width=0.1929in]
{Kpositivecross.eps}}%
\end{array}%
\right) $

\bigskip

$\chi _{q;s,t}\left(
\raisebox{-0.0796in}
{\includegraphics[height=0.2491in,width=0.2335in]
{Knegativecross.eps}}\right) $

\bigskip

$=\left(
\begin{array}{ccccccccc}
\raisebox{-0.0796in}
{\includegraphics[height=0.2032in,width=0.1929in]
{negativecross.eps}} & 0 & 0 & 0 & 0 & 0 & 0 & 0 & 0 \\
0 & 0 &
\raisebox{-0.0796in}
{\includegraphics[height=0.2032in,width=0.1929in]
{right-positivecross.eps}} & 0 & 0 & 0 & 0 & 0 & 0 \\
0 &
\raisebox{-0.0796in}
{\includegraphics[height=0.2032in,width=0.1929in]
{left-positivecross.eps}} & -p%
\raisebox{-0.0796in}
{\includegraphics[height=0.2197in,width=0.1609in]
{upper-lower-resolvecorss.eps}}+pst%
\raisebox{-0.0796in}
{\includegraphics[height=0.2266in,width=0.1781in]
{left-right-vertical-resolvecorss.eps}} & 0 & 0 & 0 & 0 & 0 & ptaq^{-1}%
\raisebox{-0.0796in}
{\includegraphics[height=0.2266in,width=0.1678in]
{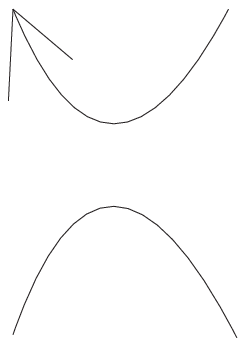}} \\
0 & 0 & 0 &
\raisebox{-0.0796in}
{\includegraphics[height=0.2032in,width=0.1929in]
{upper-negativecross.eps}} & 0 & 0 & 0 & 0 & 0 \\
0 & 0 & 0 & 0 & 0 & 0 &
\raisebox{-0.0796in}
{\includegraphics[height=0.2032in,width=0.1929in]
{rightlower-cross.eps}} & 0 & 0 \\
0 & 0 & 0 & 0 & 0 & 0 & 0 &
\raisebox{-0.0796in}
{\includegraphics[height=0.2032in,width=0.1929in]
{rightupper-cross.eps}} & 0 \\
0 & 0 & 0 & 0 &
\raisebox{-0.0796in}
{\includegraphics[height=0.2032in,width=0.1929in]
{leftlower-cross.eps}} & 0 & -p%
\raisebox{-0.0796in}
{\includegraphics[height=0.2197in,width=0.1609in]
{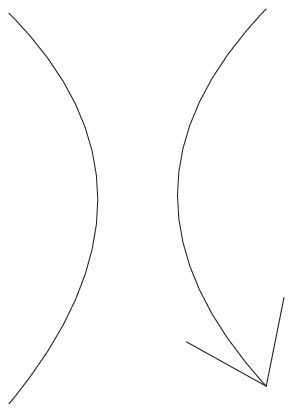}} & 0 & 0 \\
0 & 0 & 0 & 0 & 0 &
\raisebox{-0.0796in}
{\includegraphics[height=0.2032in,width=0.1929in]
{leftupper-cross.eps}} & 0 & -p%
\raisebox{-0.0796in}
{\includegraphics[height=0.2205in,width=0.1531in]
{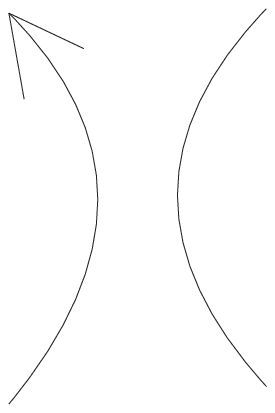}} & 0 \\
0 & 0 & pa%
\raisebox{-0.0796in}
{\includegraphics[height=0.2266in,width=0.1756in]
{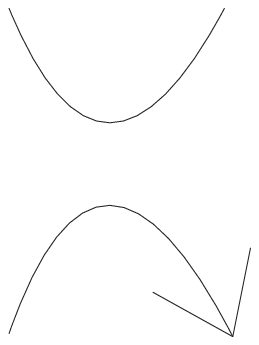}} & 0 & 0 & 0 & 0 & 0 &
\raisebox{-0.0796in}
{\includegraphics[height=0.2032in,width=0.1929in]
{Knegativecross.eps}}%
\end{array}%
\right) $

\bigskip

$\chi _{q;s,t}\left(
\raisebox{-0.0796in}
{\includegraphics[height=0.2681in,width=0.1842in]
{Kresolvecorss.eps}}\right) =\left(
\begin{array}{ccccccccc}
\raisebox{-0.0796in}
{\includegraphics[height=0.2681in,width=0.2136in]
{resolvecorss.eps}} & 0 & 0 & 0 & 0 & 0 & 0 & 0 & 0 \\
0 &
\raisebox{-0.0796in}
{\includegraphics[height=0.2681in,width=0.2032in]
{lower-upper-resolvecorss.eps}} & 0 & 0 & 0 & 0 & 0 & 0 & 0 \\
0 & 0 &
\raisebox{-0.0796in}
{\includegraphics[height=0.2681in,width=0.1946in]
{upper-lower-resolvecorss.eps}} & 0 & 0 & 0 & 0 & 0 & 0 \\
0 & 0 & 0 &
\raisebox{-0.0796in}
{\includegraphics[height=0.2681in,width=0.1842in]
{upper-uper-resolvecorss.eps}} & 0 & 0 & 0 & 0 & 0 \\
0 & 0 & 0 & 0 &
\raisebox{-0.0796in}
{\includegraphics[height=0.2681in,width=0.2032in]
{lower-no-resolvecorss.eps}} & 0 & 0 & 0 & 0 \\
0 & 0 & 0 & 0 & 0 &
\raisebox{-0.0796in}
{\includegraphics[height=0.2681in,width=0.1842in]
{no-upper-resolvecorss.eps}} & 0 & 0 & 0 \\
0 & 0 & 0 & 0 & 0 & 0 &
\raisebox{-0.0796in}
{\includegraphics[height=0.2681in,width=0.1946in]
{no-lower-resolvecorss.eps}} & 0 & 0 \\
0 & 0 & 0 & 0 & 0 & 0 & 0 &
\raisebox{-0.0796in}
{\includegraphics[height=0.2681in,width=0.1842in]
{upper-no-resolvecorss.eps}} & 0 \\
0 & 0 & 0 & 0 & 0 & 0 & 0 & 0 &
\raisebox{-0.0796in}
{\includegraphics[height=0.2681in,width=0.1842in]
{Kresolvecorss.eps}}%
\end{array}%
\right) $

\bigskip

$\chi _{q;s,t}\left(
\raisebox{-0.0502in}
{\includegraphics[height=0.1513in,width=0.3312in]
{lower.eps}}\right) =\left(
\begin{array}{c}
0 \\
t^{-1}a^{-1}\text{ }%
\raisebox{-0.0502in}
{\includegraphics[height=0.1513in,width=0.3442in]
{rightlower.eps}} \\
aq^{-1}\text{ }%
\raisebox{-0.0502in}
{\includegraphics[height=0.1513in,width=0.3312in]
{leftlower.eps}} \\
0 \\
0 \\
0 \\
0 \\
0 \\
t^{-1}%
\raisebox{-0.0502in}
{\includegraphics[height=0.1513in,width=0.3312in]
{lower.eps}}%
\end{array}%
\right) $, $\chi _{q;s,t}\left(
\raisebox{-0.0502in}
{\includegraphics[height=0.1574in,width=0.2646in]
{top.eps}}\right) =\left(
\begin{array}{ccccccccc}
0 & a^{-1}q\text{ }%
\raisebox{-0.0502in}
{\includegraphics[height=0.1565in,width=0.2681in]{lefttop.eps}} & ta\text{ }%
\raisebox{-0.0502in}
{\includegraphics[height=0.1574in,width=0.2646in]{righttop.eps}} & 0 & 0 & 0
& 0 & 0 & t\text{ }%
\raisebox{-0.0502in}
{\includegraphics[height=0.1574in,width=0.2646in]
{top.eps}}%
\end{array}%
\right) $

\bigskip

\begin{theorem}
\label{Thm Kauffman 2} The map $\chi _{q;s,t}$ extends uniquely to a
covariant functor from $\mathcal{K}_{q;st^{2}}$ to $\mathcal{HK}_{q;s,t}$.
\end{theorem}

\begin{proof}
To prove this we should check relations. Composing elementary diagrams we
have:

\bigskip

$\chi _{q;s,t}\left(
\begin{array}{c}
\raisebox{-0.0796in}
{\includegraphics[height=0.2491in,width=0.2335in]
{Kpositivecross.eps}} \\
\raisebox{-0.0796in}
{\includegraphics[height=0.2491in,width=0.2335in]
{Knegativecross.eps}}%
\end{array}%
\right) $

\bigskip

$=\left(
\begin{array}{ccccccccc}
\begin{array}{c}
\raisebox{-0.0796in}
{\includegraphics[height=0.2032in,width=0.1929in]
{positivecross.eps}} \\
\raisebox{-0.0796in}
{\includegraphics[height=0.2032in,width=0.1929in]
{negativecross.eps}}%
\end{array}
& 0 & 0 & 0 & 0 & 0 & 0 & 0 & 0 \\
0 &
\begin{array}{c}
\raisebox{-0.0796in}
{\includegraphics[height=0.2032in,width=0.1929in]
{right-negativecross.eps}} \\
\raisebox{-0.0796in}
{\includegraphics[height=0.2491in,width=0.1929in]
{left-positivecross.eps}}%
\end{array}
& \ast & 0 & 0 & 0 & 0 & 0 & \ast \ast \\
0 & 0 &
\begin{array}{c}
\raisebox{-0.0796in}
{\includegraphics[height=0.2032in,width=0.1929in]
{left-negativecross.eps}} \\
\raisebox{-0.0796in}
{\includegraphics[height=0.2032in,width=0.1929in]
{right-positivecross.eps}}%
\end{array}
& 0 & 0 & 0 & 0 & 0 & 0 \\
0 & 0 & 0 &
\begin{array}{c}
\raisebox{-0.0796in}
{\includegraphics[height=0.2032in,width=0.1929in]
{upper-positivecross.eps}} \\
\raisebox{-0.0796in}
{\includegraphics[height=0.2032in,width=0.1929in]
{upper-negativecross.eps}}%
\end{array}
& 0 & 0 & 0 & 0 & 0 \\
0 & 0 & 0 & 0 &
\begin{array}{c}
\raisebox{-0.0796in}
{\includegraphics[height=0.2032in,width=0.1929in]
{rightlower-positivecross.eps}} \\
\raisebox{-0.0796in}
{\includegraphics[height=0.2032in,width=0.1929in]
{leftlower-negativecross.eps}}%
\end{array}
& 0 & p%
\begin{array}{c}
\raisebox{-0.0796in}
{\includegraphics[height=0.2205in,width=0.1678in]
{lower-no-resolvecorss.eps}} \\
\raisebox{-0.0796in}
{\includegraphics[height=0.2032in,width=0.1929in]
{rightlower-negativecross.eps}}%
\end{array}%
-p%
\begin{array}{c}
\raisebox{-0.0796in}
{\includegraphics[height=0.2032in,width=0.1929in]
{rightlower-positivecross.eps}} \\
\raisebox{-0.0796in}
{\includegraphics[height=0.2197in,width=0.1609in]
{no-lower-resolvecorss.eps}}%
\end{array}
& 0 & 0 \\
0 & 0 & 0 & 0 & 0 &
\begin{array}{c}
\raisebox{-0.0796in}
{\includegraphics[height=0.2032in,width=0.1929in]
{rightupper-positivecross.eps}} \\
\raisebox{-0.0796in}
{\includegraphics[height=0.2032in,width=0.1929in]
{leftupper-negativecross.eps}}%
\end{array}
& 0 & p%
\begin{array}{c}
\raisebox{-0.0796in}
{\includegraphics[height=0.2205in,width=0.1531in]
{no-upper-resolvecorss.eps}} \\
\raisebox{-0.0796in}
{\includegraphics[height=0.2032in,width=0.1929in]
{rightupper-negativecross.eps}}%
\end{array}%
-p%
\begin{array}{c}
\raisebox{-0.0796in}
{\includegraphics[height=0.2032in,width=0.1929in]
{rightupper-positivecross.eps}} \\
\raisebox{-0.0796in}
{\includegraphics[height=0.2205in,width=0.1531in]
{upper-no-resolvecorss.eps}}%
\end{array}
& 0 \\
0 & 0 & 0 & 0 & 0 & 0 &
\begin{array}{c}
\raisebox{-0.0796in}
{\includegraphics[height=0.2032in,width=0.1929in]
{leftlower-positivecross.eps}} \\
\raisebox{-0.0796in}
{\includegraphics[height=0.2032in,width=0.1929in]
{rightlower-negativecross.eps}}%
\end{array}
& 0 & 0 \\
0 & 0 & 0 & 0 & 0 & 0 & 0 &
\begin{array}{c}
\raisebox{-0.0796in}
{\includegraphics[height=0.2032in,width=0.1929in]
{leftupper-positivecross.eps}} \\
\raisebox{-0.0796in}
{\includegraphics[height=0.2032in,width=0.1929in]
{rightupper-negativecross.eps} }%
\end{array}
& 0 \\
0 & 0 & \ast \ast \ast & 0 & 0 & 0 & 0 & 0 &
\begin{array}{c}
\raisebox{-0.0796in}
{\includegraphics[height=0.2032in,width=0.1929in]
{Kpositivecross.eps}} \\
\raisebox{-0.0796in}
{\includegraphics[height=0.2032in,width=0.1929in]
{Knegativecross.eps}}%
\end{array}%
\end{array}%
\right) $

\bigskip

Matrix elements which are denotes by $\ast $, $\ast \ast $ and $\ast \ast
\ast $ are easy to compute:

\bigskip

$\ast =p%
\begin{array}{c}
\raisebox{-0.0796in}
{\includegraphics[height=0.2205in,width=0.1678in]
{lower-upper-resolvecorss.eps}} \\
\raisebox{-0.0796in}
{\includegraphics[height=0.2032in,width=0.1929in]
{right-positivecross.eps}}%
\end{array}%
-ps^{-1}t^{-1}%
\begin{array}{c}
\raisebox{-0.0796in}
{\includegraphics[height=0.2266in,width=0.1652in]
{right-left-vertical-resolvecorss.eps}} \\
\raisebox{-0.0796in}
{\includegraphics[height=0.2032in,width=0.1929in]
{right-positivecross.eps}}%
\end{array}%
-p%
\begin{array}{c}
\raisebox{-0.0796in}
{\includegraphics[height=0.2032in,width=0.1929in]
{right-negativecross.eps}} \\
\raisebox{-0.0796in}
{\includegraphics[height=0.2681in,width=0.1946in]
{upper-lower-resolvecorss.eps}}%
\end{array}%
+pst%
\begin{array}{c}
\raisebox{-0.0796in}
{\includegraphics[height=0.2032in,width=0.1929in]
{right-negativecross.eps}} \\
\raisebox{-0.0796in}
{\includegraphics[height=0.2664in,width=0.2076in]
{left-right-vertical-resolvecorss.eps}}%
\end{array}%
-p^{2}%
\begin{array}{c}
\raisebox{-0.0796in}
{\includegraphics[height=0.2266in,width=0.1652in]
{right-no-vertical-resolvecorss.eps}} \\
\raisebox{-0.0796in}
{\includegraphics[height=0.2266in,width=0.1756in]
{no-right-vertical-resolvecorss.eps}}%
\end{array}%
$

\bigskip

$=p%
\raisebox{-0.0796in}
{\includegraphics[height=0.2032in,width=0.1929in]
{right-positivecross.eps}}-ps^{-1}%
\raisebox{-0.0796in}
{\includegraphics[height=0.2655in,width=0.205in]
{right-right-vertical-resolvecorss.eps}}-p%
\raisebox{-0.0796in}
{\includegraphics[height=0.2032in,width=0.1929in]
{right-negativecross.eps}}+ps%
\raisebox{-0.0796in}
{\includegraphics[height=0.2655in,width=0.205in]
{right-right-vertical-resolvecorss.eps}}-p^{2}%
\raisebox{-0.0796in}
{\includegraphics[height=0.2655in,width=0.205in]
{right-right-vertical-resolvecorss.eps}}(1+\frac{s-s^{-1}}{q-q^{-1}})$

\bigskip

$=p(%
\raisebox{-0.0796in}
{\includegraphics[height=0.2032in,width=0.1929in
] {right-positivecross.eps}}-%
\raisebox{-0.0796in}
{\includegraphics[height=0.2032in,width=0.1929in]
{right-negativecross.eps}})-ps^{-1}%
\raisebox{-0.0796in}
{\includegraphics[height=0.2655in,width=0.205in]
{right-right-vertical-resolvecorss.eps}}+ps%
\raisebox{-0.0796in}
{\includegraphics[height=0.2655in,width=0.205in]
{right-right-vertical-resolvecorss.eps}}-p^{2}%
\raisebox{-0.0796in}
{\includegraphics[height=0.2655in,width=0.205in]
{right-right-vertical-resolvecorss.eps}}(1+\frac{s-s^{-1}}{q-q^{-1}})$

\bigskip

$=p^{2}%
\raisebox{-0.0796in}
{\includegraphics[height=0.2655in,width=0.205in]
{right-right-vertical-resolvecorss.eps}}-ps^{-1}%
\raisebox{-0.0796in}
{\includegraphics[height=0.2655in,width=0.205in]
{right-right-vertical-resolvecorss.eps}}+ps%
\raisebox{-0.0796in}
{\includegraphics[height=0.2655in,width=0.205in]
{right-right-vertical-resolvecorss.eps}}-p^{2}%
\raisebox{-0.0796in}
{\includegraphics[height=0.2655in,width=0.205in]
{right-right-vertical-resolvecorss.eps}}(1+\frac{s-s^{-1}}{q-q^{-1}})$

\bigskip

$=0$

\bigskip

$\ast \ast =ptaq^{-1}%
\begin{array}{c}
\raisebox{-0.0796in}
{\includegraphics[height=0.2032in,width=0.1929in]
{right-negativecross.eps}} \\
\raisebox{-0.0796in}
{\includegraphics[height=0.2266in,width=0.1678in]
{left-no-vertical-resolvecorss.eps}}%
\end{array}%
-pa^{-1}%
\begin{array}{c}
\raisebox{-0.0796in}
{\includegraphics[height=0.2266in,width=0.1652in]
{right-no-vertical-resolvecorss.eps}} \\
\raisebox{-0.0796in}
{\includegraphics[height=0.2032in,width=0.1929in]
{Knegativecross.eps}}%
\end{array}%
=paq^{-1}%
\raisebox{-0.0796in}
{\includegraphics[height=0.2266in,width=0.1652in]
{right-no-vertical-resolvecorss.eps}}-pa^{-1}s%
\raisebox{-0.0796in}
{\includegraphics[height=0.2266in,width=0.1652in]
{right-no-vertical-resolvecorss.eps}}=0$

\bigskip

$\ast \ast \ast =-pt^{-1}a^{-1}q%
\begin{array}{c}
\raisebox{-0.0796in}
{\includegraphics[height=0.2266in,width=0.1652in]
{no-left-vertical-resolvecorss.eps}} \\
\raisebox{-0.0796in}
{\includegraphics[height=0.2032in,width=0.1929in]
{right-positivecross.eps}}%
\end{array}%
+pa%
\begin{array}{c}
\raisebox{-0.0796in}
{\includegraphics[height=0.2032in,width=0.1929in]
{Kpositivecross.eps}} \\
\raisebox{-0.0796in}
{\includegraphics[height=0.2266in,width=0.1756in]
{no-right-vertical-resolvecorss.eps}}%
\end{array}%
=-pa^{-1}q%
\raisebox{-0.0796in}
{\includegraphics[height=0.2266in,width=0.1756in]
{no-right-vertical-resolvecorss.eps}}+pas^{-1}%
\raisebox{-0.0796in}
{\includegraphics[height=0.2266in,width=0.1756in]
{no-right-vertical-resolvecorss.eps}}=0$

\bigskip

Thus we proved

$\chi _{q;s,t}\left(
\begin{array}{c}
\raisebox{-0.0796in}
{\includegraphics[height=0.2491in,width=0.2335in]
{Kpositivecross.eps}} \\
\raisebox{-0.0796in}
{\includegraphics[height=0.2491in,width=0.2335in]
{Knegativecross.eps}}%
\end{array}%
\right) =\chi _{q;s,t}\left(
\raisebox{-0.0796in}
{\includegraphics[height=0.2681in,width=0.1842in]
{Kresolvecorss.eps}}\right) $

\bigskip

By a similar computation, we get

\bigskip

$\chi _{q;s,t}\left(
\begin{array}{c}
\raisebox{-0.0796in}
{\includegraphics[height=0.2491in,width=0.2335in]
{Knegativecross.eps}} \\
\raisebox{-0.0796in}
{\includegraphics[height=0.2491in,width=0.2335in]
{Kpositivecross.eps}}%
\end{array}%
\right) =\chi _{q;s,t}\left(
\raisebox{-0.0796in}
{\includegraphics[height=0.2681in,width=0.1842in]
{Kresolvecorss.eps}}\right) $

\bigskip

Then we check the skein relation as follows

\bigskip

$\chi _{q;s,t}\left(
\raisebox{-0.0796in}
{\includegraphics[height=0.2491in,width=0.2335in]
{Kpositivecross.eps}}\right) -\chi _{q;s,t}\left(
\raisebox{-0.0796in}
{\includegraphics[height=0.2491in,width=0.2335in]
{Knegativecross.eps}}\right) $

\bigskip

$=\left(
\begin{array}{ccccccccc}
\begin{array}{c}
\raisebox{-0.0796in}
{\includegraphics[height=0.2032in,width=0.1929in]
{positivecross.eps}}- \\
\raisebox{-0.0796in}
{\includegraphics[height=0.2032in,width=0.1929in]
{negativecross.eps} }%
\end{array}
& 0 & 0 & 0 & 0 & 0 & 0 & 0 & 0 \\
0 &
\begin{array}{c}
p%
\raisebox{-0.0796in}
{\includegraphics[height=0.2205in,width=0.1678in]
{lower-upper-resolvecorss.eps}}- \\
ps^{-1}t^{-1}%
\raisebox{-0.0796in}
{\includegraphics[height=0.2266in,width=0.1652in]
{right-left-vertical-resolvecorss.eps}}%
\end{array}
&
\begin{array}{c}
\raisebox{-0.0796in}
{\includegraphics[height=0.2032in,width=0.1929in]
{right-negativecross.eps}}- \\
\raisebox{-0.0796in}
{\includegraphics[height=0.2032in,width=0.1929in]
{right-positivecross.eps}}%
\end{array}
& 0 & 0 & 0 & 0 & 0 & -pa^{-1}%
\raisebox{-0.0796in}
{\includegraphics[height=0.2266in,width=0.1652in]
{right-no-vertical-resolvecorss.eps}} \\
&  &  &  &  &  &  &  &  \\
0 &
\raisebox{-0.0796in}
{\includegraphics[height=0.2032in,width=0.1929in]
{left-negativecross.eps}}-%
\raisebox{-0.0796in}
{\includegraphics[height=0.2032in,width=0.1929in]
{left-positivecross.eps}} &
\begin{array}{c}
p%
\raisebox{-0.0796in}
{\includegraphics[height=0.2197in,width=0.1609in]
{upper-lower-resolvecorss.eps}}- \\
pst%
\raisebox{-0.0796in}
{\includegraphics[height=0.2266in,width=0.1781in]
{left-right-vertical-resolvecorss.eps}}%
\end{array}
& 0 & 0 & 0 & 0 & 0 & -ptaq^{-1}%
\raisebox{-0.0796in}
{\includegraphics[height=0.2266in,width=0.1678in]
{left-no-vertical-resolvecorss.eps}} \\
0 & 0 & 0 &
\begin{array}{c}
\raisebox{-0.0796in}
{\includegraphics[height=0.2032in,width=0.1929in]
{upper-positivecross.eps}}- \\
\raisebox{-0.0796in}
{\includegraphics[height=0.2032in,width=0.1929in]
{upper-negativecross.eps}}%
\end{array}
& 0 & 0 & 0 & 0 & 0 \\
0 & 0 & 0 & 0 & p%
\raisebox{-0.0796in}
{\includegraphics[height=0.2205in,width=0.1678in]
{lower-no-resolvecorss.eps}} & 0 &
\begin{array}{c}
\raisebox{-0.0796in}
{\includegraphics[height=0.2032in,width=0.1929in]
{rightlower-positivecross.eps} }- \\
\raisebox{-0.0796in}
{\includegraphics[height=0.2032in,width=0.1929in]
{rightlower-negativecross.eps}}%
\end{array}
& 0 & 0 \\
0 & 0 & 0 & 0 & 0 & p%
\raisebox{-0.0796in}
{\includegraphics[height=0.2205in,width=0.1531in]
{no-upper-resolvecorss.eps}} & 0 &
\begin{array}{c}
\raisebox{-0.0796in}
{\includegraphics[height=0.2032in,width=0.1929in]
{rightupper-positivecross.eps}}- \\
\raisebox{-0.0796in}
{\includegraphics[height=0.2032in,width=0.1929in]
{rightupper-negativecross.eps}}%
\end{array}
& 0 \\
0 & 0 & 0 & 0 &
\begin{array}{c}
\raisebox{-0.0796in}
{\includegraphics[height=0.2032in,width=0.1929in]
{leftlower-positivecross.eps}}- \\
\raisebox{-0.0796in}
{\includegraphics[height=0.2032in,width=0.1929in]
{leftlower-negativecross.eps}}%
\end{array}
& 0 & p%
\raisebox{-0.0796in}
{\includegraphics[height=0.2197in,width=0.1609in]
{no-lower-resolvecorss.eps}} & 0 & 0 \\
0 & 0 & 0 & 0 & 0 &
\begin{array}{c}
\raisebox{-0.0796in}
{\includegraphics[height=0.2032in,width=0.1929in]
{leftupper-positivecross.eps}}- \\
\raisebox{-0.0796in}
{\includegraphics[height=0.2032in,width=0.1929in]
{leftupper-negativecross.eps}}%
\end{array}
& 0 & p%
\raisebox{-0.0796in}
{\includegraphics[height=0.2205in,width=0.1531in]
{upper-no-resolvecorss.eps}} & 0 \\
0 & -pt^{-1}a^{-1}q%
\raisebox{-0.0796in}
{\includegraphics[height=0.2266in,width=0.1652in]
{no-left-vertical-resolvecorss.eps}} & -pa%
\raisebox{-0.0796in}
{\includegraphics[height=0.2266in,width=0.1756in]
{no-right-vertical-resolvecorss.eps}} & 0 & 0 & 0 & 0 & 0 &
\raisebox{-0.0796in}
{\includegraphics[height=0.2032in,width=0.1929in]
{Kpositivecross.eps}}-%
\raisebox{-0.0796in}
{\includegraphics[height=0.2032in,width=0.1929in]
{Knegativecross.eps}}%
\end{array}%
\right) $

\bigskip

$=\left(
\begin{array}{ccccccccc}
p%
\raisebox{-0.0796in}
{\includegraphics[height=0.2681in,width=0.2136in]
{resolvecorss.eps}} & 0 & 0 & 0 & 0 & 0 & 0 & 0 & 0 \\
0 & p%
\raisebox{-0.0796in}
{\includegraphics[height=0.2205in,width=0.1678in]
{lower-upper-resolvecorss.eps}}-ps^{-1}t^{-1}%
\raisebox{-0.0796in}
{\includegraphics[height=0.2266in,width=0.1652in]
{right-left-vertical-resolvecorss.eps}} & -p%
\raisebox{-0.0796in}
{\includegraphics[height=0.2655in,width=0.205in]
{right-right-vertical-resolvecorss.eps}} & 0 & 0 & 0 & 0 & 0 & -pa^{-1}%
\raisebox{-0.0796in}
{\includegraphics[height=0.2266in,width=0.1652in]
{right-no-vertical-resolvecorss.eps}} \\
&  &  &  &  &  &  &  &  \\
0 &
\raisebox{-0.0796in}
{\includegraphics[height=0.2032in,width=0.1929in]
{left-negativecross.eps}}-%
\raisebox{-0.0796in}
{\includegraphics[height=0.2491in,width=0.1929in]
{left-positivecross.eps}} & p%
\raisebox{-0.0796in}
{\includegraphics[height=0.2197in,width=0.1609in]
{upper-lower-resolvecorss.eps}}-pst%
\raisebox{-0.0796in}
{\includegraphics[height=0.2266in,width=0.1781in]
{left-right-vertical-resolvecorss.eps}} & 0 & 0 & 0 & 0 & 0 & -ptaq^{-1}%
\raisebox{-0.0796in}
{\includegraphics[height=0.2266in,width=0.1678in]
{left-no-vertical-resolvecorss.eps}} \\
0 & 0 & 0 & p%
\raisebox{-0.0796in}
{\includegraphics[height=0.2681in,width=0.1842in]
{upper-uper-resolvecorss.eps} } & 0 & 0 & 0 & 0 & 0 \\
0 & 0 & 0 & 0 & p%
\raisebox{-0.0796in}
{\includegraphics[height=0.2205in,width=0.1678in]
{lower-no-resolvecorss.eps}} & 0 & 0 & 0 & 0 \\
0 & 0 & 0 & 0 & 0 & p%
\raisebox{-0.0796in}
{\includegraphics[height=0.2205in,width=0.1531in]
{no-upper-resolvecorss.eps}} & 0 & 0 & 0 \\
0 & 0 & 0 & 0 & 0 & 0 & p%
\raisebox{-0.0796in}
{\includegraphics[height=0.2197in,width=0.1609in]
{no-lower-resolvecorss.eps}} & 0 & 0 \\
0 & 0 & 0 & 0 & 0 & 0 & 0 & p%
\raisebox{-0.0796in}
{\includegraphics[height=0.2205in,width=0.1531in]
{upper-no-resolvecorss.eps}} & 0 \\
0 & -pt^{-1}a^{-1}q%
\raisebox{-0.0796in}
{\includegraphics[height=0.2266in,width=0.1652in]
{no-left-vertical-resolvecorss.eps}} & -pa%
\raisebox{-0.0796in}
{\includegraphics[height=0.2266in,width=0.1756in]
{no-right-vertical-resolvecorss.eps}} & 0 & 0 & 0 & 0 & 0 & p(%
\raisebox{-0.0796in}
{\includegraphics[height=0.2681in,width=0.1842in]
{Kresolvecorss.eps}}-%
\raisebox{-0.0796in}
{\includegraphics[height=0.2672in,width=0.1929in]
{vertical-resolvecorss.eps}})%
\end{array}%
\right) $

\bigskip

$=(q-q^{-1})\left(
\begin{array}{ccccccccc}
\raisebox{-0.0796in}
{\includegraphics[height=0.2681in,width=0.2136in]
{resolvecorss.eps}} & 0 & 0 & 0 & 0 & 0 & 0 & 0 & 0 \\
0 &
\raisebox{-0.0796in}
{\includegraphics[height=0.2681in,width=0.2032in]
{lower-upper-resolvecorss.eps}} & 0 & 0 & 0 & 0 & 0 & 0 & 0 \\
0 & 0 &
\raisebox{-0.0796in}
{\includegraphics[height=0.2681in,width=0.1946in]
{upper-lower-resolvecorss.eps}} & 0 & 0 & 0 & 0 & 0 & 0 \\
0 & 0 & 0 &
\raisebox{-0.0796in}
{\includegraphics[height=0.2681in, width=0.1842in]
{upper-uper-resolvecorss.eps}} & 0 & 0 & 0 & 0 & 0 \\
0 & 0 & 0 & 0 &
\raisebox{-0.0796in}
{\includegraphics[height=0.2681in,width=0.2032in]
{lower-no-resolvecorss.eps}} & 0 & 0 & 0 & 0 \\
0 & 0 & 0 & 0 & 0 &
\raisebox{-0.0796in}
{\includegraphics[height=0.2681in,width=0.1842in]
{no-upper-resolvecorss.eps}} & 0 & 0 & 0 \\
0 & 0 & 0 & 0 & 0 & 0 &
\raisebox{-0.0796in}
{\includegraphics[height=0.2681in,width=0.1946in]
{no-lower-resolvecorss.eps}} & 0 & 0 \\
0 & 0 & 0 & 0 & 0 & 0 & 0 &
\raisebox{-0.0796in}
{\includegraphics[height=0.2681in,width=0.1842in]
{upper-no-resolvecorss.eps}} & 0 \\
0 & 0 & 0 & 0 & 0 & 0 & 0 & 0 &
\raisebox{-0.0796in}
{\includegraphics[height=0.2681in,width=0.1842in]
{Kresolvecorss.eps}}%
\end{array}%
\right) $

\bigskip

$-(q-q^{-1})\left(
\begin{array}{ccccccccc}
0 & 0 & 0 & 0 & 0 & 0 & 0 & 0 & 0 \\
0 & s^{-1}t^{-1}%
\raisebox{-0.0796in}
{\includegraphics[height=0.2672in,width=0.1929in]
{right-left-vertical-resolvecorss.eps}} &
\raisebox{-0.0796in}
{\includegraphics[height=0.2655in,width=0.205in]
{right-right-vertical-resolvecorss.eps}} & 0 & 0 & 0 & 0 & 0 & a^{-1}%
\raisebox{-0.0796in}
{\includegraphics[height=0.2672in,width=0.1929in]
{right-no-vertical-resolvecorss.eps}} \\
&  &  &  &  &  &  &  &  \\
0 &
\raisebox{-0.0796in}
{\includegraphics[height=0.2664in,width=0.1954in]
{left-left-vertical-resolvecorss.eps}} & st%
\raisebox{-0.0796in}
{\includegraphics[height=0.2664in,width=0.2076in]
{left-right-vertical-resolvecorss.eps}} & 0 & 0 & 0 & 0 & 0 & taq^{-1}%
\raisebox{-0.0796in}{\includegraphics[height=0.2664in,width=0.1954in]
{left-no-vertical-resolvecorss.eps}} \\
0 & 0 & 0 & 0 & 0 & 0 & 0 & 0 & 0 \\
0 & 0 & 0 & 0 & 0 & 0 & 0 & 0 & 0 \\
0 & 0 & 0 & 0 & 0 & 0 & 0 & 0 & 0 \\
0 & 0 & 0 & 0 & 0 & 0 & 0 & 0 & 0 \\
0 & 0 & 0 & 0 & 0 & 0 & 0 & 0 & 0 \\
0 & t^{-1}a^{-1}q%
\raisebox{-0.0796in}
{\includegraphics[height=0.2672in,width=0.1929in]
{no-left-vertical-resolvecorss.eps}} & a%
\raisebox{-0.0796in}
{\includegraphics[height=0.2655in,width=0.205in]
{no-right-vertical-resolvecorss.eps}} & 0 & 0 & 0 & 0 & 0 &
\raisebox{-0.0796in}
{\includegraphics[height=0.2672in,
width=0.1929in ] {vertical-resolvecorss.eps}}%
\end{array}%
\right) $

\bigskip

$=(q-q^{-1})\chi _{q;s,t}\left(
\raisebox{-0.0796in}
{\includegraphics[height=0.2681in,width=0.1842in]
{Kresolvecorss.eps}}-%
\raisebox{-0.0796in}
{\includegraphics[height=0.2672in,width=0.1929in]
{vertical-resolvecorss.eps}}\right) $

\bigskip

Framing moves can be checked similarly. For example, for one of the
nontrivial components of the image of the left move in (\ref{Kauffman Kink})
we have:

\bigskip

$%
\raisebox{-0.1003in}
{\includegraphics[height=0.3217in,width=0.3027in]
{positivecross.eps}}_{%
\raisebox{-0.1203in}
{\includegraphics[height=0.1513in,width=0.3442in]
{rightlower.eps}}}^{%
\raisebox{0.0803in}
{\includegraphics[height=0.1565in,width=0.2681in]
{lefttop.eps}}}$ $s^{-1}t^{-1}+(p$ $%
\raisebox{-0.1003in}
{\includegraphics[height=0.3494in,width=0.2612in]
{lower-upper-resolvecorss.eps}}-ps^{-1}t^{-1}$ $%
\raisebox{-0.1003in}
{\includegraphics[height=0.3598in,width=0.2569in]
{right-left-vertical-resolvecorss.eps}})_{%
\raisebox{-0.1203in}
{\includegraphics[height=0.1513in,width=0.3442in]
{leftlower.eps}}}^{%
\raisebox{0.0803in}
{\includegraphics[height=0.1565in,width=0.2681in]
{righttop.eps}}}$ $st+p%
\raisebox{-0.0796in}
{\includegraphics[height=0.2205in,width=0.1678in]
{lower-no-resolvecorss.eps}}_{%
\raisebox{-0.1203in}
{\includegraphics[height=0.1513in,width=0.3442in]
{lower.eps}}}^{%
\raisebox{0.0803in}
{\includegraphics[height=0.1565in,width=0.2681in]
{top.eps}}}$ $tt^{-1}$

\bigskip

$=%
\raisebox{-0.2006in}
{\includegraphics[height=0.5336in,width=0.3511in]
{positivekink.eps}}$ $s^{-1}t^{-1}+(q-q^{-1})$ $%
\raisebox{-0.2006in}
{\includegraphics[height=0.5007in,width=0.6737in]
{linecircle.eps}}$ $st-(q-q^{-1})$ $%
\raisebox{-0.2006in}
{\includegraphics[height=0.5587in,width=0.1894in]
{resolvekink.eps}}+(q-q^{-1})%
\raisebox{-0.2006in}
{\includegraphics[height=0.5007in,width=0.6737in]
{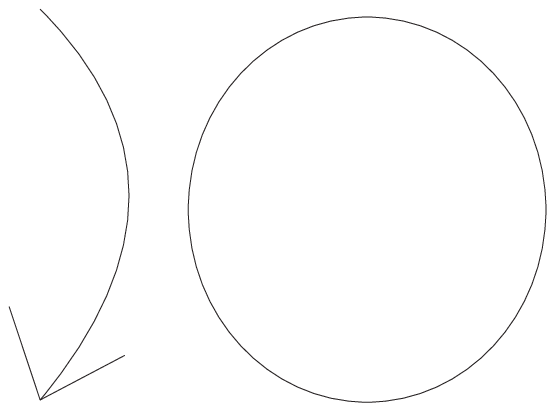}}$

\bigskip

$=%
\raisebox{-0.2006in}
{\includegraphics[height=0.5587in,width=0.1894in]
{resolvekink.eps}}$ $s^{-1}+%
\raisebox{-0.2006in}
{\includegraphics[height=0.5587in,width=0.1894in]
{resolvekink.eps}}$ $(q-q^{-1})\frac{t-t^{-1}}{q-q^{-1}}st-(q-q^{-1})$ $%
\raisebox{-0.2006in}
{\includegraphics[height=0.5587in,width=0.1894in]
{resolvekink.eps}}+%
\raisebox{-0.2006in}
{\includegraphics[height=0.5587in,width=0.1894in]
{resolvekink.eps}}(q-q^{-1})[1+\frac{s-s^{-1}}{q-q^{-1}}]$

\bigskip

$=%
\raisebox{-0.2006in}
{\includegraphics[height=0.5587in,width=0.1894in]
{resolvekink.eps}}$ $(s^{-1}+st^{2}-s-(q-q^{-1})+(q-q^{-1})+s-s^{-1})$

\bigskip

$=st^{2}$ $%
\raisebox{-0.2006in}
{\includegraphics[height=0.5587in,width=0.1894in]
{resolvekink.eps}}$

\bigskip

$p%
\raisebox{-0.0796in}
{\includegraphics[height=0.2205in,width=0.1531in]
{no-upper-resolvecorss.eps}}_{%
\raisebox{-0.1203in}
{\includegraphics[height=0.1513in,width=0.3442in]
{leftlower.eps}}}^{%
\raisebox{0.0803in}
{\includegraphics[height=0.1565in,width=0.2681in]
{righttop.eps}}}$ $st+%
\raisebox{-0.0796in}
{\includegraphics[height=0.2032in,width=0.1929in]
{Kpositivecross.eps}}_{%
\raisebox{-0.1203in}
{\includegraphics[height=0.1513in,width=0.3442in]
{lower.eps}}}^{%
\raisebox{0.0803in}
{\includegraphics[height=0.1565in,width=0.2681in]
{top.eps}}}$ $tt^{-1}$

\bigskip

$=(q-q^{-1})$ $%
\raisebox{-0.2006in}
{\includegraphics[height=0.5007in,width=0.6737in]
{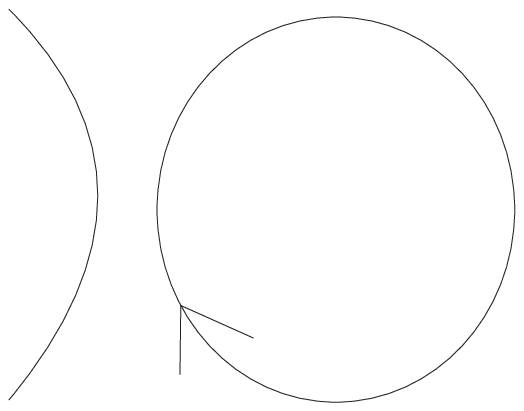}}$ $st+%
\raisebox{-0.2006in}
{\includegraphics[height=0.5336in,width=0.3503in]
{Kpositivekink.eps}}$

\bigskip

$=(q-q^{-1})\frac{t-t^{-1}}{q-q^{-1}}$ $st$ $%
\raisebox{-0.2006in}
{\includegraphics[height=0.5613in,width=0.147in]
{Kresolvekink.eps}}+s$ $%
\raisebox{-0.2006in}
{\includegraphics[height=0.5613in,width=0.147in]
{Kresolvekink.eps}}$

\bigskip

$=st^{2}$ $%
\raisebox{-0.2006in}
{\includegraphics[height=0.5613in,width=0.147in]
{Kresolvekink.eps}}$

\bigskip

and

\bigskip

$%
\raisebox{-0.0796in}
{\includegraphics[height=0.2032in,width=0.1929in]
{upper-positivecross.eps}}_{%
\raisebox{-0.1203in}
{\includegraphics[height=0.1513in,width=0.3442in]
{leftlower.eps}}}^{%
\raisebox{0.0803in}
{\includegraphics[height=0.1565in,width=0.2681in]
{righttop.eps}}}$ $st=st^{2}$ $%
\raisebox{-0.2214in}
{\includegraphics[height=0.6244in,width=0.2084in]
{resolvekinkreverse.eps}}$

\bigskip

Remaining relations can be checked similarly.
\end{proof}

When $s=\pm q^{N-1}$ and $t=q^{n-k}$, where $N=2k+1$ for $so_{2k+1}$ and $%
N=2k$ for $so_{2k}$ and $sp_{2k}$, the category $\mathcal{K}_{q;st^{2}}$ has
the quotient category which is naturally equivalent to the category of $%
U_{q}(g_{n})$-modules. Here $+$ is for $so$ and $-$ is for $sp$. For these
values of $s$ and $t$ the category $\mathcal{HK}_{q;s,t}$ is naturally
equivalent to the category of $U_{q}(g_{n})$-modules regarded as $%
U_{q}(g_{k})\otimes U_{q}(sl_{n-k})$-modules and functor $\chi _{q;s,t}$
becomes the restriction functor.


\begin{thebibliography}{99}
\bibitem{HOMFLY} P.~Freyd, D.~Yetter, J.~Hoste, W.B.R.~Lichorish, K.~Millet,
A.~Ocneanu, A new polynomial invariant of knots and links, \textit{Bull.
Amer. Math. Soc. }\textbf{12} (1985), 239-246.

\bibitem{Ja} F. Jaeger, Composition products and models for the Homfly
polynomial, \textit{L'Enseign. Math.}, \textbf{35} (1989), 323-361.

\bibitem{J} V. F. R. Jones, On knot invariants related to some statistical
mechanics models, \textit{Pacific J. of Math.}, \textbf{137} (1989) no.2,
311-334, Main results of the paper first appeared as notes of Atyiah's
seminar in 1986.

\bibitem{Kau} L. H.~Kauffman. An invariant of regular isotopy, \textit{%
Trans. Amer. Math. Soc.}, \textbf{318} (1990), 417-471.

\bibitem{KV} L. H. Kauffman, P. Vogel, Link polynomials and a graphical
calculus, \textit{J. Knot Theory Ramif.}, \textbf{1} (1992), no. 1, 59-104.

\bibitem{R} N. Yu, Reshetikhin, Quantized universal enveloping algebras, the
Yang-Baxter equation and invariants of links, \textit{LOMI preprint},
E-4-87, E-17-87, (1987). Scanned copy is posted at http://math.berkeley.edu/%
\symbol{126}reshetik/

\bibitem{RT} N. Yu.~Reshetikhin, V. G.~Turaev, Ribbon graphs and their
invariants derived from quantum groups, \textit{Comm. Math. Phys}, \textbf{%
127} (1990), 1-26.

\bibitem{T} V. Turaev: The Yang-Baxter equation and invariants of links,
\textit{Invent. Math.}, \textbf{92} (1988), 527-553.

\bibitem{Wu1} H. Wu, Colored Morton-Franks-Williams inequalities, \textit{%
Int. Math. Res. Notices.}, \textbf{20} (2013), 4734-4757, arXiv:1102.0586.

\bibitem{Wu2} H. Wu, On the Kauffman-Vogel and the Murakami-Ohtsuki-Yamada
Graph Polynomials, \textit{J. Knot Theory Ramif.}, \textbf{21} 1250098
(2012);. arXiv:1107.5333.
\end{thebibliography}
\end{document}